\documentclass[a4paper,11pt]{article}

\usepackage{a4wide}
\usepackage[utf8]{inputenc}
\usepackage{amsmath}
\usepackage{amsfonts}
\usepackage{amssymb}
\usepackage{mathrsfs}
\usepackage{amsthm}
\usepackage{color}
\usepackage{stmaryrd}
\usepackage[hidelinks]{hyperref}

\newtheorem{thm}{Theorem}[section]
\newtheorem{lemma}[thm]{Lemma}
\newtheorem{corollary}[thm]{Corollary}

\theoremstyle{definition}
\newtheorem{definition}[thm]{Definition}

\theoremstyle{remark}
\newtheorem{remark}[thm]{Remark}

\newcommand{\dx}[0]{\,\mathrm{d}}

\newcommand{\fbsde}[0]{}
\newcommand{\esssup}[0]{\mathrm{ess}\,\mathrm{sup}}

\title{An FBSDE approach to the Skorokhod embedding problem for Gaussian processes with non-linear drift}
\author{Alexander Fromm\thanks{A.F. was financially supported by the DFG Research Training Group 1845 ``Stochastic Analysis with Applications in Biology, Finance and Physics''.}
\\ Institut f{\"u}r Mathematik \\ Technische Universit{\"a}t Berlin
\and Peter Imkeller
\\ Institut f{\"u}r Mathematik \\ Humboldt-Universit{\"a}t zu Berlin 
\and David J. Pr{\"o}mel\thanks{D.J.P. gratefully acknowledges the support by the DFG Research Training Group 1845 ``Stochastic Analysis with Applications in Biology, Finance and Physics'' and by the Swiss National Foundation (Grand No. 200021$\_$163014).}
\\ Department of Mathematics \\ ETH Z\"urich}

\begin{document}

\maketitle

\begin{abstract}
  We solve the Skorokhod embedding problem for a class of Gaussian processes including Brownian motion with non-linear drift. Our approach relies on solving an associated strongly coupled system of Forward Backward Stochastic Differential Equations (FBSDEs), and investigating the regularity of the obtained solution. For this purpose we extend the existence, uniqueness and regularity theory of so called decoupling fields for Markovian FBSDE to a setting in which the coefficients are only locally Lipschitz continuous.
\end{abstract}

\noindent\emph{MSC 2010:} Primary: 60G40, 60H30; Secondary: 93E20. \smallskip\\
\emph{Key words and phrases: BMO process, BSDE, Decoupling field, Forward backward stochastic differential equation, FBSDE, Skorokhod embedding, Variational differentiation.}


\section{Introduction}

The Skorokhod embedding problem (SEP) stimulates research in probability theory now for over 50 years. The classical goal of the SEP consists in finding, for a given Brownian motion $W$ and a probability measure $\nu$, a stopping time $\tau$ such that $W_\tau$ possesses the law $\nu$. It was first formulated and solved by Skorokhod \cite{Skorokhod1961,Skorokhod1965} in 1961. Since then there appeared many different constructions for the stopping time $\tau$ and generalizations of the original problem in the literature. Just to name some of the most famous solutions to the SEP we refer to Root \cite{Root1969}, Rost \cite{Rost1971} and Az\'ema-Yor \cite{Azema1979}. A comprehensive survey can be found in \cite{Obloj2004}.

Recently, the Skorokhod embedding raised additional interest because of its new applications in financial mathematics, as for instance to obtain model-independent bounds on lookback options \cite{Hobson1998} or on options on variance \cite{Carr2010,Cox2013,Oberhauser2013}. An introduction to this close connection of the Skorokhod embedding problem and robust financial mathematics can be found in \cite{Hobson2011}. 

In this paper we construct a solution to the Skorokhod embedding problem for Gaussian process $G$ of the form
\begin{equation*}
  G_t:=G_0+\int_{0}^t\alpha_s\dx s+\int_{0}^t\beta_s\dx W_s,
\end{equation*}
where $G_0\in\mathbb{R}$ is a constant and $\alpha,\beta\colon[0,\infty)\to\mathbb{R}$ are suitable functions. Especially, this class of processes includes  Brownian motions with non-linear drift. The SEP for Brownian motion with linear drift was first solved in the technical report \cite{Hall1968} and 30 years later again in \cite{Grandits2000} and \cite{Peskir2000}. Techniques developed in these works can be extended to time-homogeneous diffusions, as done in \cite{Pedersen2001}, and can be seen as generalization of the Az\'ema-Yor solution. However, to the best of our knowledge there exists no solution so far for the case of a Brownian motion with non-linear drift.

The spirit of our approach is related to the one by Bass \cite{Bass1983}, who employed martingale representation to find an alternative solution of the SEP for the Brownian motion. This approach was further developed for the Brownian motion with linear drift in \cite{Ankirchner2008} and for time-homogeneous diffusion in \cite{Ankirchner2015}. It rests upon the observation that the SEP may be viewed as the weak version of a stochastic control problem: the goal is to steer $G$ in such a way that it takes the distribution of a prescribed law, which in case of zero drift is closely related to the martingale representation of a random variable with this law. We therefore propose in this paper to formulate and solve the SEP for $G$ in terms of a fully coupled Forward Backward Stochastic Differential Equation (FBSDE).

In general terms, the dynamics of a system of FBSDE is expressed by the equations
\begin{align*}
  X_s&=X_{0}+\int_{0}^s\mu(r,X_r,Y_r,Z_r)\dx r+\int_{0}^s\sigma(r,X_r,Y_r,Z_r)\dx W_r,\\
  Y_t&=\xi(X_T)-\int_{t}^{T}f(r,X_r,Y_r,Z_r)\dx r-\int_{t}^{T}Z_r\dx W_r,\quad t\in[0,T],
\end{align*}
with coefficient functions $\mu, \sigma$ of the forward part, terminal condition $\xi$ and driver $f$ of the backward component. In recent decades the theory of FBSDE with its close connection to quasi-linear partial differential equations and their viscosity solutions has been propagated extensively, in particular in its numerous areas of applications as stochastic control and mathematical finance (see \cite{ElKaroui1997} or \cite{Peng1999}).

There are mainly three methods to show the existence of a solution for a system of FBSDE: the \emph{contraction method} \cite{Antonelli1993,Pardoux1999}, the \emph{four step scheme} \cite{Ma1994} and the \emph{method of continuation} \cite{Hu1995,Yong1997,Ma1999}. As a unified approach, \cite{Ma2015} (see also \cite{Delarue2002}) designed the theory of decoupling fields for FBSDE, which was refined and extended to a multidimensional setting in \cite{Fromm2013,Fromm2015}. It can be seen as an extension of the contraction method. In our approach of the SEP via FBSDE, we shall focus on the subclass of Markovian ones for which all involved coefficient functions $(\xi,(\mu,\sigma,f))$ are deterministic. We, however, have to allow for not globally, but only locally Lipschitz continuous coefficients $(\mu,\sigma,f)$ in the control variable $z$, and therefore to develop an existence, uniqueness and regularity theory for FBSDE in this framework. 

Equipped with these tools we solve the FBSDE system resulting from the SEP. We first construct a weak solution, i.e. we obtain a Gaussian process of the above form and an integrable random time such that, stopped at this time, the process possesses the given distribution $\nu$. Under suitable regularity on the given measure $\nu$ and the process, this construction will be carried over to the originally given Gaussian process $G$. This solves the SEP for $G$. 

The paper is organized as follows: in Section \ref{sec:SEP} we relate the SEP to a fully coupled system of FBSDE, and in Section \ref{sec:dec} we establish general results for decoupling fields of FBSDE. The Skorokhod embedding problem is solved in Section \ref{sec:solving}, in its weak and in its strong version. Section \ref{sec:appendix} recalls some auxiliary results for BMO processes.

\section{An FBSDE approach to the Skorokhod embedding problem}\label{sec:SEP}

We consider a filtered probability space $(\Omega, \mathcal{F}, (\mathcal{F}_t)_{t \in [0,\infty)},\mathbb{P})$ large enough to carry a one-dimensional Brownian motion $W$ and with $\mathcal{F}:=\sigma\left(\bigcup_{t=0}^\infty\mathcal{F}_t\right)$. The filtration $(\mathcal{F}_t)$ is assumed to be generated by the Brownian motion and to be augmented by $\mathbb{P}$-null sets. 

Let us start by formulating the \textit{Skorokhod embedding problem} \textbf{(SEP)} in the modified version: For a given probability measure $\nu$ on $\mathbb{R}$ and a Gaussian process $G$ on $[0,\infty)$ of the form
\begin{equation}\label{eq:gaussian}
  G_t:=G_0+\int_{0}^t\alpha_s\dx s+\int_{0}^t\beta_s\dx W_s,
\end{equation}
where $G_0\in\mathbb{R}$ is a constant and $\alpha,\beta\colon[0,\infty)\to\mathbb{R}$ are deterministic measurable processes such that $\int_0^t|\alpha_s|\dx s+\int_0^t\beta^2_s\dx s<\infty$ for all $t\geq 0$, find an integrable $(\mathcal{F}_t)$-stopping time $\tau$ together with a starting point $c\in\mathbb{R}$ such that $c+G_\tau$ has the law $\nu$.

In order to have a truly stochastic problem $\beta$ should not vanish and $\nu$ should not be a Dirac measure. In fact, we will assume that $\beta$ is bounded away from zero later on.

Our method of solving this problem is based on the observation that it may be viewed as the weak version of a \textit{stochastic control problem}: We want to steer $G$ in such a way that it takes the distribution of a prescribed law. The spirit of our approach is related to an approach to the original Skorokhod embedding problem by Bass \cite{Bass1983} that was later extended to the Brownian motion with linear drift in \cite{Ankirchner2008}. The procedure of both papers can be briefly summarized and divided into the following four steps.

\begin{enumerate}
  \item Construct a function $g \colon \mathbb{R} \to \mathbb{R}$ such that $g(W_1)$ has the given law $\nu$.
  \item Use the martingale representation property of the Brownian motion for $\alpha \equiv 0$ and $\beta\equiv 1$ or BSDE techniques for $\alpha\equiv \kappa \neq 0$ and $\beta\equiv 1$ to solve
       \begin{equation}\label{eq:bsde}
         Y_t = g(W_1) - \kappa \int_t^1 Z^2_s \dx s - \int_t^1 Z_s \dx W_s, \quad t\in [0,1].
       \end{equation}
  \item Apply the random time-change of Dambis, Dubins and Schwarz in the quadratic variation scale $\int_0^. Z_s^2\dx s$ to transform the martingale $\int_0^. Z_s \dx W_s$ into a Brownian motion $B$. This also provides a random time $\tilde \tau := \int_0^1 Z_s^2\dx s$ fulfilling $B_{\tilde \tau} + \kappa \tilde \tau+Y_0= g(W_1) $, which is why $B_{\tilde \tau}+ \kappa \tilde \tau+Y_0$ has the law $\nu$.
  \item Show that $\tilde\tau$ is a stopping time with respect to the filtration generated by $B$ through an explicit characterization using the unique solution of an ordinary differential equation. With this description transform the embedding with respect to $B$ into one with respect to the original Brownian motion $W$ to obtain the stopping time $\tau$ as the analogue to $\tilde{\tau}$.
\end{enumerate}

The {\bf first step} of the algorithm just sketched is fairly easy. Let $F \colon \mathbb{R} \to [0,1]$ such that $F(x):= \nu((-\infty, x]) $ is the cumulative distribution function associated with $\nu$ and define $F^{-1}\colon(0,1)\rightarrow\mathbb{R}$ via $F^{-1}(y):= \inf \{x \in \mathbb{R}\, :\, F (x) \geq y \}$. Denoting by $\Phi$ the distribution function of the standard normal distribution, we define $g \colon \mathbb{R} \to \mathbb{R}$ by $g (x):= F^{-1} (\Phi(x))$. It is straightforward to prove that $g$ has the following properties.

\begin{lemma}
  The function $g$ is measurable and non-decreasing. Moreover, if $\nu$ is not a Dirac measure, then $g$ is not identically constant and $g(W_1)$ has the law $\nu$.
\end{lemma}

\begin{proof}
  Since $\Phi$ and $F^{-1}$ are measurable and non-decreasing, their composition $g$ is also measurable and non-decreasing. 
  
  Clearly, $g$ can only be constant if $F^{-1}$ is constant, which can only happen if $F$ assumes values in $\{0,1\}$. This only happens in case $\nu$ is a Dirac measure. In order to see that $g(W_1)$ has the law $\nu$, note that $\mathbb{P}(g(W_1) \leq x ) = \mathbb{P}(W_1\leq \Phi^{-1}(F(x)))= \Phi(\Phi^{-1}(F(x)))=F(x)$ for $x \in \mathbb{R}$.
\end{proof}

Since we want to require as little regularity as possible for the processes involved, we use the concept of weak differentiability. We recall that a measurable $f\colon \Omega \times \mathbb{R}^n\to \mathbb{R}$ is \textit{weakly differentiable} if there exists a mapping $\frac{\dx}{\dx\lambda}f\colon\Omega \times \mathbb{R}^n\to\mathbb{R}^{1\times n}$ such that 
\begin{equation*}
\int_{\mathbb{R}^n}\varphi(\lambda)\frac{\dx}{\dx\lambda}f(\omega,\lambda)\dx\lambda=
-\int_{\mathbb{R}^n}f(\omega,\lambda)\frac{\dx}{\dx\lambda}\varphi(\lambda)\dx\lambda,
\end{equation*}
for any smooth test function $\varphi\colon\mathbb{R}^n\to\mathbb{R}$ with compact support, for almost all $\omega \in \Omega$.

Now we define a measurable function $\hat\delta\colon [0,\infty)\to\mathbb{R}$ via $\hat\delta(t):=G_0+\int_0^t\alpha_s\dx s$ such that $X_t=\hat\delta(t)+\int_0^t\beta_s\dx W_s$. Obviously, $\hat\delta$ is weakly differentiable (if setting it to $G_0=\hat\delta(0)$ outside of $[0,\infty)$). Conversely, for every weakly differentiable function $\hat\delta\colon[0,\infty)\to\mathbb{R}$ we can set $G_0:=\hat\delta(0)$ and $\alpha_s:=\hat\delta'(s)$.

Furthermore, define $H\colon [0,\infty)\to[0,\infty)$ via $ H(t):=\int_0^t\beta^2_s\dx s$. Note that $H$ is weakly differentiable, monotonically increasing and starts at $0$. If we assume that $\beta$ is bounded away from $0$, $H$ becomes strictly increasing and invertible such that the inverse function $H^{-1}$ is monotonically increasing and Lipschitz continuous. In this case we can define $\delta:=\hat\delta\circ H^{-1}$. Notice, if $\beta\equiv 1$, then $H=\mathrm{Id}$ and thus $\delta=\hat\delta$.

For the {\bf second step} we assume that $\beta$ is bounded away from $0$ and observe that the random time change, which turns the martingale $\int_0^\cdot Z_s \dx W_s$ into a Gaussian process of the form $\int_{0}^\cdot\beta_s\dx B_s$ simultaneously turns the scale process $\int_0^. Z_s^2 \dx s$ into $\int_0^\cdot \beta^2_s\dx s=H$. This means we have to modify the classical martingale representation of $g(W_1)$ to
\begin{equation*}
  g(W_1) + \hat\delta \bigg(H^{-1}\bigg(\int_0^1 Z^2_s \dx s \bigg)\bigg) - \mathbb{E}\bigg[g(W_1) + \hat\delta \bigg(H^{-1}\bigg(\int_0^1 Z^2_s \dx s \bigg)\bigg) \bigg] =  \int_0^1 Z_s \dx W_s,
\end{equation*}
which amounts to finding a solution $(Y,Z)$ to the equation
\begin{equation}\label{eq:fbsde}
  Y_t = g(W_1) - \delta\bigg(\int_0^1 Z_s^2 \dx s\bigg) - \int_t^1 Z_s \dx W_s,  \quad t \in [0,1].
\end{equation}
For $\delta(t) \equiv 0$ this would be just the usual martingale representation with respect to the Brownian motion. Also for a linear drift
$\delta(t) = \kappa t$ and $\beta\equiv 1$ equation \eqref{eq:fbsde} can be rewritten as
\begin{equation*}
 \tilde Y_t:= Y_t + \kappa \int_0^t Z_s^2 \dx s = g(W_1) - \kappa  \int_t^1 Z_s^2 \dx s  - \int_t^1 Z_s \dx W_s, \quad t \in [0,1],
\end{equation*}
which is exactly the BSDE \eqref{eq:bsde} related to the SEP as stated in \cite{Ankirchner2008}. In the case of a Brownian motion with general drift equation \eqref{eq:fbsde} would be a BSDE with time-delayed terminal condition. Unfortunately, the theory of BSDE with time-delay as introduced by Delong and Imkeller \cite{Delong2010} and extended by Delong \cite{Delong2012} for time-delayed terminal conditions reaches its limits in our situation. Alternatively, we will understand equation \eqref{eq:fbsde} as an FBSDE and develop new techniques to solve it. This will be done in Section \ref{sec:dec} and \ref{sec:solving}.

Before we tackle the solvability of equation \eqref{eq:fbsde}, we show that it really leads to the desired result in the {\bf third step} of our algorithm. To be mathematically rigorous we introduce the space $\mathbb{S}^2(\mathbb{R})$ of all progressively measurable processes $Y \colon \Omega \times [0,1] \to \mathbb{R}$ satisfying $\sup_{t \in [0,1]}\mathbb{E}[ \vert Y_t \vert^2 ] <\infty$, and the space $\mathbb{H}^2(\mathbb{R})$ of all progressively measurable processes $Z \colon \Omega \times [0,1] \to \mathbb{R}$ satisfying $\mathbb{E}[\int_0^1 \vert Z_t \vert^2 \dx t] <\infty$, where $\vert \cdot \vert $ denotes the Euclidean norm on $\mathbb{R}$.

In the entire paper we assume that $\beta$ is {\bf bounded away from $0$}, i.e. $\inf_{s\in[0,\infty)}|\beta_s|>0$.

\begin{lemma}\label{lem:weak}
  Suppose that $(Y,Z)\in \mathbb{S}^2(\mathbb{R}) \times \mathbb{H}^2(\mathbb{R})$ is a solution of \eqref{eq:fbsde}.
  Then there exist a Brownian motion $B$ and a random time $\tilde \tau$ with $\mathbb{E}[\tilde \tau] < \infty$
  such that
  \begin{equation*}
    Y_0+G_0+\int_0^{\tilde \tau}\alpha_s \dx s+\int_0^{\tilde\tau}\beta_s\dx B_s = g(W_1).
  \end{equation*}
\end{lemma}

\begin{proof}
  Note that $Y$ is a martingale with quadratic variation process $\int_0^t Z^2_s \dx s$ for $t \in [0,1]$ since  $Z \in \mathbb{H}^2(\mathbb{R})$. Now choose another Brownian motion $\tilde B$ which is independent of $Y$. If necessary we extend our probability space such that it accommodates the Brownian motion $\tilde B$. Set $\tilde \tau:= H^{-1}\big(\int_0^1 Z^2_s \dx s\big)$, and define the time-change of the type of Dambis, Dubins and Schwarz by
  \begin{equation*}
    \sigma_r := \begin{cases}
                      \inf \left\{ t \geq 0 \,:\, \int_0^t Z^2_s \dx s  >\int_0^r \beta^2_s \dx s \right\}, & \text{if } 0 \leq r < \tilde\tau, \\
                      1, & \text{if } r \geq \tilde\tau.
                    \end{cases}
  \end{equation*}
  Observe that the condition $r<\tilde\tau$ is equivalent to $\int_0^r \beta^2_s \dx s<\int_0^1 Z^2_s \dx s$. Since $Y_{\sigma_r}$ is a continuous martingale with quadratic variation $H(r)=\int_0^r \beta^2_s \dx s$, we can define a Brownian motion $B$ by
  \begin{equation*}
    B_r:= \tilde B_r-\tilde B_{r \wedge \tilde\tau}+\int_0^{r\wedge\tilde\tau}\frac{1}{\beta_s}\dx Y_{\sigma_s}, \quad 0 \leq r < \infty.
  \end{equation*}
  We find
  \begin{equation*}
    \int_0^{\tilde\tau}\beta_s\dx B_s +\hat\delta(\tilde\tau)+Y_0 = Y_1-Y_0 + \delta\bigg(\int_0^1 Z^2_s\dx s\bigg) + Y_0 = g(W_1),
  \end{equation*}
  and further $\mathbb{E}[\tilde\tau]= \mathbb{E}\big[H^{-1}\big(\int_0^1 Z^2_s \dx s\big)\big] < \infty$,
  where we used that $Z \in \mathbb{H}^2(\mathbb{R})$ and $H^{-1}$ is Lipschitz continuous.
\end{proof}

As an immediate consequence of the previous lemma we observe the following fact: If we have a solution $(Y,Z)\in \mathbb{S}^2(\mathbb{R}) \times \mathbb{H}^2(\mathbb{R})$ of equation \eqref{eq:fbsde}, we obtain a \textit{weak solution} to the Skorokhod embedding problem, i.e. a Gaussian process of the form \eqref{eq:gaussian}, a starting point $c$, and an integrable random time such that our process stopped at this time possesses a given distribution.

At a first glance equation \eqref{eq:fbsde} might look easy. We, however, have to deal with a fully coupled FBSDE which in addition possesses a not globally Lipschitz continuous coefficient in the forward component.

\section{Decoupling fields for fully coupled FBSDEs}\label{sec:dec}

The theory of FBSDEs, closely connected to the theory of quasi-linear partial differential equations and their viscosity solutions, receives its general interest from numerous areas of application among which stochastic control and mathematical finance are the most vivid ones in recent decades (see \cite{ElKaroui1997} or \cite{Peng1999}). Owing to their general significance, we treat the theory of FBSDEs and their decoupling fields in a more general framework than might be needed to obtain a solution to our equation \eqref{eq:fbsde}. 

Although in Section \ref{sec:markovian} we will focus on the Markovian case, which means that all involved coefficients are purely deterministic, let us dwell in a more general setting first.

\subsection{General decoupling fields}

For a fixed finite time horizon $T>0$, we consider a complete filtered probability space $(\Omega,\mathcal{F},(\mathcal{F}_t)_{t\in[0,T]},\mathbb{P})$, where $\mathcal{F}_0$ contains all null sets, $(W_t)_{t\in[0,T]}$ is a $d$-dimensional Brownian motion independent of $\mathcal{F}_0$, and $\mathcal{F}_t:=\sigma(\mathcal{F}_0,(W_s)_{s\in [0,t]})$ with $\mathcal{F}:=\mathcal{F}_T$. The dynamics of an FBSDE is classically given by
\begin{align*}
  X_s&=X_{0}+\int_{0}^s\mu(r,X_r,Y_r,Z_r)\dx r+\int_{0}^s\sigma(r,X_r,Y_r,Z_r)\dx W_r,\\
  Y_t&=\xi(X_T)-\int_{t}^{T}f(r,X_r,Y_r,Z_r)\dx r-\int_{t}^{T}Z_r\dx W_r,
\end{align*}
for $s,t \in [0,T]$ and $X_0 \in \mathbb{R}^n$, where $(\xi,(\mu,\sigma,f))$ are measurable functions such that 
\begin{align*}
  \xi &\colon\Omega\times\mathbb{R}^n \to \mathbb{R}^m, &
  \mu &\colon [0,T]\times\Omega\times\mathbb{R}^n\times\mathbb{R}^m\times\mathbb{R}^{m\times d}\to \mathbb{R}^n,\\
  \sigma&\colon [0,T]\times\Omega\times\mathbb{R}^n\times\mathbb{R}^m\times\mathbb{R}^{m\times d}\to \mathbb{R}^{n\times d},&
  f&\colon [0,T]\times\Omega\times\mathbb{R}^n\times\mathbb{R}^m\times\mathbb{R}^{m\times d}\to \mathbb{R}^m,
\end{align*}
for $d,n,m\in\mathbb{N}$. Throughout the whole section $\mu$, $\sigma$ and $f$ are assumed to be progressively measurable with respect to $(\mathcal{F}_t)_{t\in[0,T]}$.

A decoupling field comes with an even richer structure than just a classical solution.

\begin{definition}\label{def:decoupling field}
  Let $t\in[0,T]$. A function $u\colon [t,T]\times\Omega\times\mathbb{R}^n\to\mathbb{R}^m$ with $u(T,\cdot)=\xi$ a.e. is called \emph{decoupling field} for $\fbsde (\xi,(\mu,\sigma,f))$ on $[t,T]$ if for all $t_1,t_2\in[t,T]$ with $t_1\leq t_2$ and any $\mathcal{F}_{t_1}$-measurable $X_{t_1}\colon\Omega\to\mathbb{R}^n$ there exist progressively measurable processes $(X,Y,Z)$ on $[t_1,t_2]$ such that
  \begin{align}\label{eq:decoupling}
    X_s&=X_{t_1}+\int_{t_1}^s\mu(r,X_r,Y_r,Z_r)\dx r+\int_{t_1}^s\sigma(r,X_r,Y_r,Z_r)\dx W_r,&\nonumber\\
    Y_s&=Y_{t_2}-\int_{s}^{t_2}f(r,X_r,Y_r,Z_r)\dx r-\int_{s}^{t_2}Z_r\dx W_r,&
    Y_s&=u(s,X_s),
  \end{align}
  for all $s\in[t_1,t_2]$. In particular, we want all integrals to be well-defined.
\end{definition}

Some remarks about this definition are in place.
\begin{itemize}
  \item The first equation in \eqref{eq:decoupling} is called the \emph{forward equation}, the second the \emph{backward equation} and the third will be referred to as the \emph{decoupling condition}.
  \item Note that, if $t_2=T$, we get $Y_T=\xi(X_T)$ a.s. as a consequence of the decoupling condition together with $u(T,\cdot)=\xi$. At the same time $Y_T=\xi(X_T)$ together with decoupling condition implies $u(T,\cdot)=\xi$ a.e.
  \item If $t_2=T$ we can say that a triplet $(X,Y,Z)$ solves the FBSDE, meaning that it satisfies the forward and the backward equation, together with $Y_T=\xi(X_T)$. This relationship $Y_T=\xi(X_T)$ is referred to as the \emph{terminal condition}. 
\end{itemize}

In contrast to classical solutions of FBSDEs, decoupling fields on different intervals can be pasted together.

\begin{lemma}[\cite{Fromm2015}, Lemma 2.1.2]\label{glue}
  Let $u$ be a decoupling field for $\fbsde (\xi,(\mu,\sigma,f))$ on $[t,T]$ and $\tilde{u}$ be a decoupling field for $\fbsde (u(t,\cdot),(\mu,\sigma,f))$ on $[s,t]$, for $0\leq s<t<T$. Then, the map $\hat{u}$ given by $\hat{u}:=\tilde{u}\mathbf{1}_{[s,t]}+u\mathbf{1}_{(t,T]}$ is a decoupling field for $\fbsde (\xi,(\mu,\sigma,f))$ on $[s,T]$.
\end{lemma}

We want to remark that, if $u$ is a decoupling field and $\tilde{u}$ is a modification of $u$, i.e. for each $s\in[t,T]$ the functions $u(s,\omega,\cdot)$ and $\tilde{u}(s,\omega,\cdot)$ coincide for almost all $\omega\in\Omega$, then $\tilde{u}$ is also a decoupling field to the same problem. Hence, $u$ could also be referred to as a class of modifications and a progressively measurable representative exists if the decoupling field is Lipschitz continuous in $x$ (Lemma 2.1.3 in \cite{Fromm2015}).

For the following we need to fix briefly further notation.

Let $I\subseteq [0,T]$ be an interval and $u: I\times\Omega\times\mathbb{R}^n\rightarrow \mathbb{R}^m$ a map such that $u(s,\cdot)$ is measurable for every $s\in I$. We define
\begin{equation*}
  L_{u,x}:=\sup_{s\in I}\inf\{L\geq 0\,|\,\textrm{for a.a. }\omega\in\Omega: |u(s,\omega,x)-u(s,\omega,x')|\leq L|x-x'|\textrm{ for all }x,x'\in\mathbb{R}^n\},
\end{equation*}
where $\inf \emptyset:=\infty$. We also set $ L_{u,x}:=\infty$ if $u(s,\cdot)$ is not measurable for every $s\in I$. One can show that $L_{u,x}<\infty$ is equivalent to $u$ having a modification which is truly Lipschitz continuous in $x\in\mathbb{R}^n$.

We denote by $L_{\sigma,z}$ the Lipschitz constant of $\sigma$ w.r.t. the dependence on the last component $z$ and w.r.t. the Frobenius norms on $\mathbb{R}^{m\times d}$ and $\mathbb{R}^{n\times d}$. We set $L_{\sigma,z}=\infty$ if $\sigma$ is not Lipschitz continuous in $z$. 

By $L_{\sigma,z}^{-1}=\frac{1}{L_{\sigma,z}}$ we mean $\frac{1}{L_{\sigma,z}}$ if $L_{\sigma,z}>0$ and $\infty$ otherwise.

For an integrable real valued random variable $F$ the expression $\mathbb{E}_t[F]$ refers to $\mathbb{E}[F|\mathcal{F}_t]$, while $\mathbb{E}_{\hat{t},\infty}[F]$ refers to $\esssup\,\mathbb{E}[F|\mathcal{F}_t]$, which might be $\infty$, but is always well defined as the infimum of all constants $c\in[-\infty,\infty]$ such that $\mathbb{E}[F|\mathcal{F}_t]\leq c$ a.s. Additionally, we write $\|F\|_\infty$ for the essential supremum of $|F|$.

In practice it is important to have explicit knowledge about the regularity of $(X,Y,Z)$. For instance, it is important to know in which spaces the processes live, and how they react to changes in the initial value. 

\begin{definition}\label{def:regularity decoupling}
  Let $u\colon [t,T]\times\Omega\times\mathbb{R}^n\to\mathbb{R}^m$ be a decoupling field to $\fbsde(\xi,(\mu,\sigma,f))$.
  \begin{enumerate}
   \item We say $u$ to be \emph{weakly regular} if $L_{u,x}<L_{\sigma,z}^{-1}$ and $\sup_{s\in[t,T]}\|u(s,\cdot,0)\|_{\infty}<\infty$.
   \item A weakly regular decoupling field $u$ is called \emph{strongly regular} if for all fixed $t_1,t_2\in[t,T]$, $t_1\leq t_2,$ the processes $(X,Y,Z)$ arising in \eqref{eq:decoupling}  are a.e unique and satisfy
   \begin{equation}\label{strongregul1}
      \sup_{s\in [t_1,t_2]}\mathbb{E}_{t_1,\infty}[|X_s|^2]+\sup_{s\in [t_1,t_2]}\mathbb{E}_{t_1,\infty}[|Y_s|^2]
      +\mathbb{E}_{t_1,\infty}\left[\int_{t_1}^{t_2}|Z_s|^2\dx s\right]<\infty,
   \end{equation}
   for each constant initial value $X_{t_1}=x\in\mathbb{R}^n$. In addition they are required to be measurable as functions of $(x,s,\omega)$ and even weakly differentiable w.r.t. $x\in\mathbb{R}^n$ such that for every $s\in[t_1,t_2]$ the mappings $X_s$ and $Y_s$ are measurable functions of $(x,\omega)$ and even weakly differentiable w.r.t. $x$ such that
   \begin{align}\label{strongregul2}
     &\esssup_{x\in\mathbb{R}^n}\sup_{v\in S^{n-1}}\sup_{s\in [t_1,t_2]}\mathbb{E}_{t_1,\infty}\left[\left|\frac{\dx}{\dx x}X_s\right|^2_v\right]<\infty, \nonumber\allowdisplaybreaks\\
     &\esssup_{x\in\mathbb{R}^n}\sup_{v\in S^{n-1}}\sup_{s\in [t_1,t_2]}\mathbb{E}_{t_1,\infty}\left[\left|\frac{\dx}{\dx x}Y_s\right|^2_v\right]<\infty, \nonumber\\
     &\esssup_{x\in\mathbb{R}^n}\sup_{v\in S^{n-1}}\mathbb{E}_{t_1,\infty}\left[\int_{t_1}^{t_2}\left|\frac{\dx}{\dx x}Z_s\right|^2_v\dx s\right]<\infty.
   \end{align}
   \item We say that a decoupling field on $[t,T]$ is \emph{strongly regular} on a subinterval $[t_1,t_2]\subseteq[t,T]$ if $u$ restricted to $[t_1,t_2]$ is a strongly regular decoupling field for $\fbsde(u(t_2,\cdot),(\mu,\sigma,f))$.
   \end{enumerate}
\end{definition}

Under suitable conditions a rich existence, uniqueness and regularity theory for decoupling fields can be developed. We will summarize the main results, which are proven in Chapter 2 of \cite{Fromm2015}:
\smallskip\\
{\bf Assumption (SLC):} $(\xi,(\mu,\sigma,f))$ satisfies \emph{standard Lipschitz conditions} \textup{(SLC)} if
\begin{enumerate}
  \item $(\mu,\sigma,f)$ are Lipschitz continuous in $(x,y,z)$ with Lipschitz constant $L$,
  \item $\left\|\left(|\mu|+|f|+|\sigma|\right)(\cdot,\cdot,0,0,0)\right\|_{\infty}<\infty$,
  \item $\xi\colon \Omega\times\mathbb{R}^n\to \mathbb{R}^m$ is measurable such that $\|\xi(\cdot,0)\|_{\infty}<\infty$ and $L_{\xi,x}<L_{\sigma,z}^{-1}$.
\end{enumerate}

\begin{thm}[\cite{Fromm2015}, Theorem 2.2.1]\label{locexist}
  Suppose $(\xi,(\mu,\sigma,f))$ satisfies \textup{(SLC)}. Then there exists a time $t\in[0,T)$ such that $\fbsde (\xi,(\mu,\sigma,f))$ has a unique (up to modification) decoupling field $u$ on $[t,T]$ with $L_{u,x}<L_{\sigma,z}^{-1}$ and $\sup_{s\in [t,T]}\|u(s,\cdot,0)\|_{\infty}<\infty$.
\end{thm}

A brief discussion of existence and uniqueness of classical solutions can be found in Remark 2.2.4 in \cite{Fromm2015}. For later reference we give the following remarks (cf. Remarks 2.2.2 and 2.2.3 in \cite{Fromm2015}).

\begin{remark}\label{hchoice}
  It can be observed from the proof that the supremum of all $h=T-t$, with $t$ satisfying the properties required in Theorem \ref{locexist} can be bounded away from $0$ by a bound, which only depends on the Lipschitz constant of $(\mu,\sigma,f)$ with respect to the last $3$ components, $T$, $L_{\sigma,z}$,  $L_{\xi}$ and $L_{\xi}\cdot L_{\sigma,z}<1$, and which is monotonically decreasing in these values.

  Furthermore, we notice from the proof that our decoupling field $u$ on $[t,T]$ satisfies $L_{u(s,\cdot),x}\leq L_{\xi,x}+C(T-s)^{\frac{1}{4}}$, where $C$ is some constant which does not depend on $s\in[t,T]$. More precisely, $C$ depends only on $T$, $L$, $L_{\xi,x}$, $L_{\xi,x}L_{\sigma,z}$ and is monotonically increasing in these values.
\end{remark}

This local theory for decoupling fields can be systematically extended to global results based on fairly simple ``small interval induction'' arguments (Lemma 2.5.1 and 2.5.2 in \cite{Fromm2015}).

\begin{thm}[\cite{Fromm2015}, Corollary 2.5.3, 2.5.4 and 2.5.5]\label{uniq}
  Suppose that $(\xi,(\mu,\sigma,f))$ satisfies \textup{(SLC)}.
  \begin{enumerate}
    \item\textup{Global uniqueness:} If there are two weakly regular decoupling fields $u^{(1)},u^{(2)}$ to the corresponding problem on some interval $[t,T]$, then we have $u^{(1)}=u^{(2)}$ up to modifications.
    \item\textup{Global regularity:} If there exists a weakly regular decoupling field $u$ to this problem on some interval $[t,T]$, then $u$ is strongly regular.
    \item If there exists a weakly regular decoupling field $u$ of the corresponding FBSDE on some interval $[t,T]$, then for any initial condition $X_t=x\in\mathbb{R}^n$ there is a unique solution $(X,Y,Z)$ of the FBSDE on $[t,T]$ satisfying
    \begin{equation*}
      \sup_{s\in[t,T]}\mathbb{E}[|X_s|^2]+\sup_{s\in[t,T]}\mathbb{E}[|Y_s|^2]+\mathbb{E}\left[\int_t^T|Z_s|^2\dx s\right]<\infty.
    \end{equation*}
  \end{enumerate}
\end{thm}

\subsection{Markovian decoupling fields}\label{sec:markovian}

A system of FBSDEs given by $(\xi,(\mu,\sigma,f))$ is said to be \emph{Markovian} if these four coefficient functions are deterministic, that is, if they depend only on $(t,x,y,z)$. In the Markovian situation we can somewhat relax the Lipschitz continuity assumption and still obtain local existence together with uniqueness. What makes the Markovian case so special is the property
\begin{equation*}
  "Z_s=u_x(s,X_s)\cdot\sigma(s,X_s,Y_s,Z_s)",
\end{equation*}
which comes from the fact that $u$ will also be deterministic. This property allows us to bound $Z$ by a constant if we assume that $\sigma$ is bounded.

\begin{lemma}[\cite{Fromm2015}, Lemma 2.5.13, 2.5.14 and 2.5.15]\label{deter}
  Let $(\xi,(\mu,\sigma,f))$ be deterministic functions and satisfy \textup{(SLC)}. Suppose that there exist a weakly regular decoupling field $u$  on an interval $[t,T]$
  \begin{enumerate}
   \item The decoupling field $u$ is deterministic in the sense that it has a modification which is a function of $(r,x)\in[t,T]\times\mathbb{R}^n$ only.
   \item For an initial condition $X_{t}$ the corresponding $Z$ satisfies $\|Z\|_\infty\leq L_{u,x}\cdot\|\sigma\|_\infty$.
   
         If $\|Z\|_{\infty}<\infty$, we also have $\|Z\|_\infty\leq L_{u,x}\|\sigma(\cdot,\cdot,\cdot,0)\|_\infty (1-L_{u,x}L_{\sigma,z})^{-1}$.
   \item Assume further that $(\mu,\sigma,f)$ have linear growth in $(x,y)$ in the sense
         \begin{equation*}
           \left(|\mu|+|\sigma|+|f|\right)(t,\omega,x,y,z)\leq C\left(1+|x|+|y|\right)\quad\forall (t,x,y,z)\in[0,T]\times\mathbb{R}^{n}\times\mathbb{R}^{m}\times\mathbb{R}^{m\times d},
         \end{equation*}
         for a.a. $\omega\in\Omega$, where $C\in[0,\infty)$ is some constant.  If $u$ is a strongly regular and deterministic decoupling field, then $u$ is continuous in the sense that it has a modification which is a continuous function on $[t,T]\times\mathbb{R}^n$.
  \end{enumerate}
\end{lemma}

In the Markovian case this boundedness of $Z$ motivates the following definition, which will allow us to develop a theory for non-Lipschitz problems via truncation.

\begin{definition}
  Let $t\in[0,T]$ and let $(\xi,(\mu,\sigma,f))$ be deterministic functions. We call a function $u\colon[t,T]\times\Omega\times\mathbb{R}^n\to\mathbb{R}^m$ with $u(T,\cdot)=\xi(\cdot)$ a.s. a \emph{Markovian decoupling field} for $\fbsde (\xi,(\mu,\sigma,f))$ on $[t,T]$ if $u$ is a decoupling field in the sense of Definition \ref{def:decoupling field} and additionally $\|Z\|_\infty<\infty$.
\end{definition}

The regularity properties for Markovian decoupling fields are analogously defined as for the standard decoupling fields (cf. Definition \ref{def:regularity decoupling}), which a slightly modification for strongly regular decoupling fields.

\begin{definition}
  Let $u\colon[t,T]\times\Omega\times\mathbb{R}^n\to\mathbb{R}^m$ be a Markovian decoupling field to $\fbsde(\xi,(\mu,\sigma,f))$. We call a weakly regular $u$ \emph{strongly regular} if for all fixed $t_1,t_2\in[t,T]$, $t_1\leq t_2,$ the processes $(X,Y,Z)$ arising in the defining property of a Markovian decoupling field are a.e. unique for each \emph{constant} initial value $X_{t_1}=x\in\mathbb{R}^n$ and satisfy \eqref{strongregul1}. In addition they must be measurable as functions of $(x,s,\omega)$ and even weakly differentiable w.r.t. $x\in\mathbb{R}^n$ such that for every $s\in[t_1,t_2]$ the mappings $X_s$ and $Y_s$ are measurable functions of $(x,\omega)$, and even weakly differentiable w.r.t. $x$ such that \eqref{strongregul2} holds.
\end{definition}

Let us introduce the assumption on the coefficients for which an existence and uniqueness theory will be developed.
\medskip\\
{\bf Assumption (MLLC):} $(\xi,(\mu,\sigma,f))$ fulfills a \emph{modified local Lipschitz condition} \textup{(MLLC)} if
\begin{enumerate}
  \item  the functions $(\mu,\sigma,f)$ are deterministic,
    \begin{enumerate}
      \item Lipschitz continuous in $(x,y,z)$ on sets of the form $[0,T]\times\mathbb{R}^n\times\mathbb{R}^{m} \times B$, where $B\subset \mathbb{R}^{m\times d}$ is an arbitrary bounded set,
      \item and fulfill $\|\mu(\cdot,0,0,0)\|_\infty,\|f(\cdot,0,0,0)\|_{\infty},\|\sigma(\cdot,\cdot,\cdot,0)\|_{\infty},L_{\sigma,z}<\infty$,
    \end{enumerate}
  \item $\xi\colon \mathbb{R}^n\to \mathbb{R}^m$ satisfies $L_{\xi,x}<L_{\sigma,z}^{-1}$.
\end{enumerate}

We begin by providing a local existence result.

\begin{thm}\label{LocLip}
  Let $(\xi,(\mu,\sigma,f))$ satisfy \textup{(MLLC)}. Then there exists a time $t\in[0,T)$ such that $\fbsde (\xi,(\mu,\sigma,f))$ has a unique weakly regular Markovian decoupling field $u$ on $[t,T]$. This $u$ is also strongly regular, deterministic, continuous and satisfies $\sup_{t_1,t_2,X_{t_1}}\|Z\|_\infty<\infty$, where $t_1<t_2$ are from $[t,T]$ and $X_{t_1}$ is an initial value (see the definition of a Markovian decoupling field for the meaning of these variables).
\end{thm}

\begin{proof}
  For any constant $H>0$ let $\chi_{H}\colon\mathbb{R}^{m\times d}\to \mathbb{R}^{m\times d}$ be defined as
  \begin{equation*}
    \chi_H(z):=\mathbf{1}_{\{|z|<H\}}z+\frac{H}{|z|}\mathbf{1}_{\{|z|\geq H\}}z.
  \end{equation*}
  It is easy to check that $\chi_{H}$ is Lipschitz continuous with Lipschitz constant $L_{\chi_H}=1$ and bounded by $H$. Furthermore, we have $\chi_H(z)=z$ if $|z|\leq H$. We implement an "inner cutoff" by defining $(\mu_H,\sigma_H,f_H)$ via $\mu_H(t,x,y,z):=\mu(t,x,y,\chi_H(z))$, etc.
  
  The boundedness of $\chi_{H}$ together with its Lipschitz continuity makes $(\mu_H,\sigma_H,f_H)$ Lipschitz continuous with some Lipschitz constant $L_H$. Furthermore, $L_{\sigma_H,z}\leq L_{\sigma,z}$. Also $(\mu_H,\sigma_H,f_H)$ have linear growth in $(y,z)$ as required by Lemma \ref{deter}. According to Theorem \ref{locexist} we know that the problem given by $\fbsde (\xi,(\mu_H,\sigma_H,f_H))$ has a unique weakly regular decoupling field $u$ on some small interval $[t',T]$ where $t'\in[0,T)$. We also know that this $u$ is strongly regular, $u$ is deterministic (by Lemma \ref{deter}), and continuous (by Lemma \ref{deter}).
  
  We will show that for sufficiently large $H$ and $t\in[t',T)$ it will also be a Markovian decoupling field to the problem $\fbsde (\xi,(\mu,\sigma,f))$. By Remark \ref{hchoice} we obtain $L_{u(t,\cdot),x}\leq L_{\xi,x}+C_H(T-t)^{\frac{1}{4}}$ for all  $t\in[t',T]$, where $C_H<\infty$ is a constant which does not depend on $t\in[t',T]$. For any $t_1\in [t',T]$ and $\mathcal{F}_{t_1}$-measurable initial value $X_{t_1}$ consider the corresponding unique $(X,Y,Z)$ on $[{t_1},T]$ satisfying the forward equation, the backward equation and the decoupling condition for $\mu_H,\sigma_H,f_H$ and $u$. Using Lemma \ref{deter} we have $\|Z\|_\infty\leq L_{u,x}\|\sigma_H\|_\infty\leq L_{u,x}\left(\|\sigma(\cdot,\cdot,\cdot,0)\|_{\infty}+L_{\sigma,z}H\right)<\infty$ and thus
  \begin{align}\label{Zbound}
    \|Z\|_\infty&\leq \frac{\sup_{s\in[t_1,T]}L_{u(s,\cdot),x}\cdot\|\sigma(\cdot,\cdot,\cdot,0)\|_\infty}{1-\sup_{s\in[t_1,T]}L_{u(s,\cdot),x}L_{\sigma,z}}
     \leq \frac{\left(L_{\xi,x}+C_H(T-t_1)^{\frac{1}{4}}\right)\cdot\|\sigma(\cdot,\cdot,\cdot,0)\|_\infty}{1-L_{\xi,x}L_{\sigma,z}-L_{\sigma,z}C_H(T-t_1)^{\frac{1}{4}}}\nonumber \\
     &=\frac{L_{\xi,x}\|\sigma(\cdot,\cdot,\cdot,0)\|_\infty}{1-L_{\xi,x}L_{\sigma,z}-L_{\sigma,z}C_H(T-t_1)^{\frac{1}{4}}}+ \frac{C_H(T-t_1)^{\frac{1}{4}}\cdot\|\sigma(\cdot,\cdot,\cdot,0)\|_\infty}{1-L_{\xi,x}L_{\sigma,z}-L_{\sigma,z}C_H(T-t_1)^{\frac{1}{4}}}
  \end{align}
  for $T-t_1$ small enough. 
  
  Now we only need to choose $H$ large enough such that $\frac{L_{\xi,x}\|\sigma(\cdot,\cdot,\cdot,0)\|_\infty}{1-L_{\xi,x}L_{\sigma,z}}$ becomes smaller than $\frac{H}{4}$, and then in the second step choose $t$ close enough to $T$ such that $L_{\sigma,z}C_H(T-t)^{\frac{1}{4}}$ becomes smaller than $\frac{1}{2}\left(1-L_{\xi,x}L_{\sigma,z}\right)$ and $\frac{C_H\|\sigma(\cdot,\cdot,\cdot,0)\|_\infty (T-t)^{\frac{1}{4}}}{1-L_{\xi,x}L_{\sigma,z}}$ becomes smaller than $\frac{H}{4}$.

  Considering \eqref{Zbound} this implies that if $t_1\in[t,T]$ the process $Z$ a.e. does not leave the region in which the cutoff is "passive", i.e. the ball of radius $H$.
  Therefore, $u$ restricted to the interval $[t,T]$ is a decoupling field to $\fbsde (\xi,(\mu,\sigma,f))$, not just to $\fbsde (\xi,(\mu_H,\sigma_H,f_H))$. It is even a Markovian decoupling field due to the boundedness of $Z$. As a Markovian decoupling field it is weakly regular, because it is weakly regular as a decoupling field to $\fbsde (\xi,(\mu_H,\sigma_H,f_H))$.

  For the uniqueness we assume than there is another weakly regular Markovian decoupling field $\tilde{u}$ to $\fbsde (\xi,(\mu,\sigma,f))$ on $[t,T]$. Choose a $t_1\in [t,T]$ and an $x\in\mathbb{R}^n$ as initial condition $X_{t_1}=x$, and consider the corresponding processes $(\tilde{X},\tilde{Y},\tilde{Z})$ that satisfy the corresponding FBSDE on $[t_1,T]$, together with the decoupling condition via $\tilde{u}$. At the same time consider $(X,Y,Z)$ solving the same FBSDE on $[t_1,T]$, but associated with the Markovian decoupling field $u$. Since $\tilde{Z},Z$ are bounded, the two triplets $(\tilde{X},\tilde{Y},\tilde{Z})$ and $(X,Y,Z)$ also solve the Lipschitz FBSDE given by $\fbsde (\xi,(\mu_H,\sigma_H,f_H))$ on $[t_1,T]$ for $H$ large enough. The two conditions $\tilde{Y}_s=\tilde{u}(s,\tilde{X}_s)$ and $Y_s=u(s,X_s)$ imply by Remark 2.2.4 in \cite{Fromm2015} that both triplets are progressively measurable processes on $[t_1,T]\times\Omega$ such that
  \begin{equation*}
    \sup_{s\in[t_1,T]}\mathbb{E}_{0,\infty}\left[|X_s|^2\right]+\sup_{s\in[t_1,T]}\mathbb{E}_{0,\infty}\left[|Y_s|^2|\right]+\mathbb{E}_{0,\infty}\left[\int_{t_1}^T|Z_s|^2\dx s\right]<\infty
  \end{equation*}
  and coincide. In particular, $\tilde{u}(t_1,x)=\tilde{Y}_{t_1}=Y_{t_1}=u(t_1,x)$.

  Strong regularity of $u$ as a Markovian decoupling field to $\fbsde (\xi,(\mu,\sigma,f))$ follows directly from the above argument about uniqueness of $(X,Y,Z)$ for deterministic initial values and bounded $Z$, and the strong regularity of $u$ as decoupling field to $\fbsde (\xi,(\mu_H,\sigma_H,f_H))$.
\end{proof}

\begin{remark}\label{hchoiceM}
  We observe from the proof that the supremum of all $h=T-t$ with $t$ satisfying the hypotheses of Theorem \ref{LocLip} can be bounded away from $0$ by a bound, which only depends on
  $L_{\xi,x}$, $L_{\xi,x}\cdot L_{\sigma,z}$,  $\|\sigma(\cdot,\cdot,\cdot,0)\|_\infty$, $T$, $L_{\sigma,z}$ and the values $(L_H)_{H\in[0,\infty)}$ where $L_H$ is the Lipschitz constant of $(\mu,\sigma,f)$ on $[0,T]\times\mathbb{R}^n\times \mathbb{R}^m \times B_H$ w.r.t. to the last $3$ components, where $B_H\subset \mathbb{R}^{m\times d}$ denotes the ball of radius $H$ with center $0$. This bounded is monotonically decreasing in these values.
\end{remark}

The following natural concept introduces a type of Markovian decoupling fields for non-Lipschitz problems (non-Lipschitz in $z$), to which nevertheless Lipschitz results can be applied.

\begin{definition}
  Let $u$ be a Markovian decoupling field for $\fbsde(\xi,(\mu,\sigma,f))$.
  \begin{itemize}
    \item We call $u$ \emph{controlled in $z$} if there exists a constant $C>0$ such that for all $t_1,t_2\in[t,T]$, $t_1\leq t_2$, and all initial values $X_{t_1}$, the corresponding processes $(X,Y,Z)$ from the definition of a Markovian decoupling field satisfy $|Z_s(\omega)|\leq C$, for almost all $(s,\omega)\in[t,T]\times\Omega$. If for a fixed triplet $(t_1,t_2,X_{t_1})$ there are different choices for $(X,Y,Z)$, then all of them are supposed to satisfy the above control.
    \item We say that a Markovian decoupling field on $[t,T]$ is \emph{controlled in $z$} on a subinterval $[t_1,t_2]\subseteq[t,T]$ if $u$ restricted to $[t_1,t_2]$ is a Markovian decoupling field for $\fbsde(u(t_2,\cdot),(\mu,\sigma,f))$ that is controlled in $z$.
    \item A Markovian decoupling field $u$ on an interval $(s,T]$ is said to be \emph{controlled in $z$} if it is controlled in $z$ on every compact subinterval $[t,T]\subseteq (s,T]$ with $C$ possibly depending on $t$.
  \end{itemize}
\end{definition}

\begin{remark}\label{zconrem}
  Our Markovian decoupling field from Theorem \ref{LocLip} is obviously controlled in $Z$: consider \eqref{Zbound} together with the choice of $t\leq t_1$ made in the proof.
\end{remark}

\begin{remark}\label{deterpluscontin}
  Let $(\xi,(\mu,\sigma,f))$ satisfy \textup{(MLLC)}, and assume that we have a Markovian decoupling field $u$ on some interval $[t,T]$, which is weakly regular and controlled in $z$. Then $u$ is also a solution to a Lipschitz problem obtained through a cutoff as in Theorem \ref{LocLip}. As such it is strongly regular (Theorem \ref{uniq}) and deterministic (Lemma \ref{deter}). But Lemma \ref{deter} is also applicable, since due to the use of a cutoff we can assume the type of linear growth required there. Thus, $u$ is also continuous.
\end{remark}

\begin{lemma}\label{gluecontrolM}
  Let $(\xi,(\mu,\sigma,f))$ satisfy \textup{(MLLC)}. For $0\leq s<t<T$ let $u$ be a weakly regular Markovian decoupling field for $\fbsde (\xi,(\mu,\sigma,f))$ on $[s,T]$.
  If $u$ is controlled in $z$ on $[s,t]$ and $T-t$ is small enough as required in Theorem \ref{LocLip} resp. Remark \ref{hchoiceM}, then $u$ is controlled in $z$ on $[s,T]$.
\end{lemma}

\begin{proof}
  Clearly, $u$ is not just controlled in $z$ on $[s,t]$, but also on $[t,T]$ (with a possibly different constant), according to Remark \ref{zconrem}. Define $C$ as the maximum of these two constants. 
  
  We only need to control $Z$ by $C$ for the case $s\leq t_1\leq t\leq t_2\leq T$, the other two cases being trivial. For this purpose consider the processes $(X,Y,Z)$ on the interval $[t_1,t_2]$ corresponding to some initial value $X_{t_1}$ and fulfilling the forward equation, the backward equation and the decoupling condition. Since the restrictions of these processes to $[t_1,t]$ still fulfill these three properties we obtain $|Z_r(\omega)|\leq C$ for almost all $r\in[t_1,t]$, $\omega\in\Omega$.
  
  At the same time, if we restrict $(X,Y,Z)$ to $[t,t_2]$, we observe that these restrictions satisfy the forward equation, the backward equation and the decoupling condition for the interval $[t,t_2]$ with $X_t$ as initial value. Therefore, $|Z_r|\leq C$ holds for a.s. for $r\in [t,t_2]$.
\end{proof}

The following important result allows us to connect the \textup{(MLLC)}-case to \textup{(SLC)}.

\begin{thm}\label{controllM}
  Let be such that \textup{(MLLC)} is satisfied and assume that there exists a weakly regular Markovian decoupling field $u$ for $(\xi,(\mu,\sigma,f))$ on $[t,T]$. Then $u$ is controlled in $z$.
\end{thm}

\begin{proof}
  Let $S\subseteq [t,T]$ be the set of all times $s\in[t,T]$, s.t. $u$ is controlled in $z$ on $[t,s]$.
  \begin{itemize}
    \item Clearly $t\in S$: For the interval $[t,t]=\{t\}$ one can only choose $t_1=t_2=t$ and so $Z\colon[t,t]\times\Omega\rightarrow\mathbb{R}^{m\times d}$ is $\mathrm{d} t \otimes\mathrm{d}\mathbb{P}$-a.e. $0$, independently of the initial value $X_{t_1}$. So we can take for $C$ any positive value.
    \item Let $s\in S$ be arbitrary. According to Lemma \ref{gluecontrolM} there exists an $h>0$ s.t. $u$ is controlled in $z$ on $[t,(s+h)\wedge T]$ since  $\|u((s+h)\wedge T,\cdot)\|_\infty<\infty$ and $L_{u((s+h)\wedge T,\cdot)}<L_{\sigma,z}^{-1}$. Considering Remark \ref{hchoiceM} and the requirements  $\|u\|_\infty<\infty$, $L_{u,x}<L_{\sigma,z}^{-1}$, we can choose $h$ independently of $s$.
  \end{itemize}
  This shows $S=[t,T]$ by small interval induction (Lemma 2.5.1 and 2.5.2 in \cite{Fromm2015}).
\end{proof}

Note that Theorem \ref{controllM} implies together with Remark \ref{deterpluscontin} that a weakly regular Markovian decoupling field to an \textup{(MLLC)} problem is deterministic and continuous.

Such a $u$ will be a standard decoupling field to an \textup{(SLC)} problem if we truncate $(\mu,\sigma,f)$ appropriately. We can thereby extend the whole theory to \textup{(MLLC)} problems:

\begin{thm}[]\label{uniqM}
  Let $(\xi,(\mu,\sigma,f))$ satisfy \textup{(MLLC)}.
  \begin{enumerate}
    \item \textup{Global uniqueness:} If there are two weakly regular Markovian decoupling fields $u^{(1)},u^{(2)}$ to this problem on some interval $[t,T]$, then $u^{(1)}=u^{(2)}$.
    \item \textup{Global regularity:} If that there exists a weakly regular Markovian decoupling field $u$ to this problem on some interval $[t,T]$, then $u$ is strongly regular.
  \end{enumerate}
\end{thm}

\begin{proof}
  1. We know that $u^{(1)}$ and $u^{(2)}$ are controlled in $z$. Choose a passive cutoff (see proof of Theorem \ref{LocLip}) and apply 1. of Theorem \ref{uniq}.
  
  2. $u$ is controlled in $z$. Choose a passive cutoff (see proof of Theorem \ref{LocLip}) and apply 2. of Theorem \ref{uniq}.
\end{proof}

\begin{lemma}\label{uniqXYZM}
  Let $(\xi,\mu,\sigma,f))$ satisfy \textup{(MLLC)} and assume that there exists a weakly regular Markovian decoupling field $u$ of the corresponding FBSDE on some interval $[t,T]$.
  
  Then for any initial condition $X_t=x\in\mathbb{R}^n$ there is a unique solution $(X,Y,Z)$ of the FBSDE on $[t,T]$ such that
  \begin{equation*}
    \sup_{s\in[t,T]}\mathbb{E}[|X_s|^2]+\sup_{s\in[t,T]}\mathbb{E}[|Y_s|^2]+\|Z\|_\infty<\infty.
  \end{equation*}
\end{lemma}

\begin{proof}
  Existence follows from the fact that $u$ is also strongly regular according to 2. of Theorem \ref{uniqM} and controlled in $z$ according to Theorem \ref{controllM}. 
  
  Uniqueness follows from Corollary \ref{uniq}: Assume there are two solutions $(X,Y,Z)$ and $(\tilde X,\tilde Y,\tilde Z)$ to the FBSDE on $[t,T]$ both satisfying the aforementioned bound. But then they both solve an \textup{(SLC)}-conform FBSDE obtained through a passive cutoff. So they must coincide according to Corollary \ref{uniq}.
\end{proof}

\begin{definition}
  Let $I^{M}_{\mathrm{max}}\subseteq[0,T]$ for $\fbsde(\xi,(\mu,\sigma,f))$ be the union of all intervals $[t,T]\subseteq[0,T]$ such that there exists a weakly regular Markovian decoupling field $u$ on $[t,T]$.
\end{definition}

Unfortunately, the maximal interval might very well be open to the left. Therefore, we need to make our notions more precise in the following definitions.

\begin{definition}
  Let $0\leq t<T$.
  \begin{itemize}
    \item We call a function $u\colon(t,T]\times\mathbb{R}^n\to\mathbb{R}^m$ a Markovian decoupling field for $\fbsde(\xi,(\mu,\sigma,f))$ on $(t,T]$ if $u$ restricted to $[t',T]$ is a Markovian decoupling field for all $t'\in(t,T]$.
    \item A Markovian decoupling field $u$ on $(t,T]$ is said to be \emph{weakly regular} if $u$ restricted to $[t',T]$ is a weakly regular Markovian decoupling field for all $t'\in(t,T]$.
    \item A Markovian decoupling field $u$ on $(t,T]$ is said to be \emph{strongly regular} if $u$ restricted to $[t',T]$ is strongly regular for all $t'\in(t,T]$.
  \end{itemize}
\end{definition}

\begin{thm}[Global existence in weak form]\label{globalexistM}
  Let $(\xi,(\mu,\sigma,f))$ satisfy \textup{(MLLC)}. Then there exists a unique weakly regular Markovian decoupling field $u$ on $I^M_{\mathrm{max}}$. This $u$ is also deterministic, controlled in $z$ and strongly regular. 
  
  Moreover, either $I^{M}_{\mathrm{max}}=[0,T]$ or $I^{M}_{\mathrm{max}}=(t^{M}_{\mathrm{min}},T]$, where $0\leq t^{M}_{\mathrm{min}}<T$.
\end{thm}

\begin{proof}
  Let $t\in I^M_{\mathrm{max}}$. Obviously, there exists a Markovian decoupling field $\check{u}^{(t)}$ on $[t,T]$ satisfying $L_{\check{u}^{(t)},x}<L_{\sigma,z}^{-1}$ and $\sup_{s\in[t,T]}\|\check{u}^{(t)}(s,\cdot,0)\|_\infty<\infty$.
  $\check{u}^{(t)}$ is controlled in $z$ and strongly regular due to Theorems \ref{controllM} and \ref{uniqM}.
  We can further assume w.l.o.g. that $\check{u}^{(t)}$ is a continuous function on $[t,T]\times\mathbb{R}^n$ according to Remark \ref{deterpluscontin}.
  There is only one such $\check{u}^{(t)}$ according to Theorem \ref{uniqM}. Furthermore, for $t,t'\in I^M_{\mathrm{max}}$ the functions $\check{u}^{(t)}$ and $\check{u}^{(t')}$ coincide on $[t\vee t',T]$ because of Theorem \ref{uniqM}. 
  
  Define $u(t,\cdot):=\check{u}^{(t)}(t,\cdot)$ for all $t\in I^M_{\mathrm{max}}$. This function $u$ is a Markovian decoupling field on $[t,T]$, since it coincides with $\check{u}^{(t)}$ on $[t,T]$. Therefore, $u$ is a Markovian decoupling field on the whole interval $I^M_{\mathrm{max}}$ and satisfies $L_{u|_{[t,T]},x}<L_{\sigma,z}^{-1}$ and $\sup_{s\in[t,T]}\|u|_{[t,T]}(s,\cdot,0)\|_\infty<\infty$ for all $t\in I^M_{\mathrm{max}}$. 
  
  Uniqueness of $u$ follows directly from Theorem \ref{uniqM} applied to every interval $[t,T]\subseteq I^M_{\mathrm{max}}$. 
  
  Addressing the form of $I^M_{\mathrm{max}}$, we see that $I^M_{\mathrm{max}}=[t,T]$ with $t\in(0,T]$ is not possible: Assume otherwise. According to the above there exists a Markovian decoupling field $u$ on $[t,T]$ s.t. $L_{u,x}<L_{\sigma,z}^{-1}$ and $\sup_{s\in[t,T]}\|u(s,\cdot,0)\|_\infty<\infty$. But then $u$ can be extended a little bit to the left using Theorem \ref{LocLip} and Lemma \ref{glue}, which contradicts the definition of $I^M_{\mathrm{max}}$.
\end{proof}

The next result states that for a singularity $t^{M}_{\mathrm{min}}$ to occur $u_x$ has to "explode" at $t^M_{\mathrm{min}}$.

\begin{lemma}\label{explosionM}
  Let $(\xi,(\mu,\sigma,f))$ satisfy \textup{(MLLC)}. If $I^{M}_{\mathrm{max}}=(t^{M}_{\mathrm{min}},T]$, then $\lim_{t\downarrow t^{M}_{\mathrm{min}}}L_{u(t,\cdot),x}=L_{\sigma,z}^{-1}$, where $u$ is the Markovian decoupling field according to Theorem \ref{globalexistM}.
\end{lemma}

\begin{proof}
  We argue indirectly and assume otherwise. Then we can select times $t_n\downarrow t^M_{\mathrm{min}}$, $n\to\infty$ such that $\sup_{n\in\mathbb{N}}L_{u(t_n,\cdot),x}<L_{\sigma,z}^{-1}$. But then we may choose an $h>0$ due to Remark \ref{hchoiceM} which does not depend on $n$ and choose $n$ large enough to have $t_n-t^M_{\mathrm{min}}<h$. So $u$ can be extended to the larger interval $[(t_n-h)\vee 0,T]$ contradicting the definition of $I^M_{\mathrm{max}}$.
\end{proof}

\section{Solution to the Skorokhod embedding problem}\label{sec:solving}

In this section we present a solution to the Skorokhod embedding problem as stated in (SEP) at the beginning of Section \ref{sec:SEP} based on solutions of the associated system of FBSDEs.

\subsection{Weak solution}

Let us therefore return to our FBSDE \eqref{eq:fbsde} that can be rewritten slightly more generally as
\begin{align}\label{eq:fbsde2}
  X^{(1)}_s &= x^{(1)}+ \int_{t}^s 1 \dx W_r,  &X^{(2)}_s &= x^{(2)} +\int_{t}^s Z_r^2 \dx r, \notag \\
  Y_s       &= g(X^{(1)}_T)- \delta (X^{(2)}_T)-\int_{s}^{T}Z_r\dx W_r,  &u(s,X^{(1)}_s,X^{(2)}_s) &= Y_s,
\end{align}
for $s\in[t,T]$ and $(x^{(1)},x^{(2)})\in\mathbb{R}^2$. In particular, this FBSDE satisfies \textup{(MLLC)} and by choosing $(x^{(1)},x^{(2)}):=(0,0)$ and $T=1$ we have $X^{(1)}_1=W_1$ and $X^{(2)}_1=\int_{0}^1 Z_s^2 \dx s$, which makes the FBSDE equivalent to \eqref{eq:fbsde}.

With the general results of Section \ref{sec:markovian} at hand we are capable to solve this system of equations. In other words, we perform the {\bf second step} of our algorithm to solve the SEP.

\begin{lemma}\label{lem:ext}
  Assume that $\delta$ and $g$ are Lipschitz continuous. Then for the FBSDE \eqref{eq:fbsde2} there exists a unique weakly regular Markovian decoupling field $u$ on $[0,T]$. This $u$ is strongly regular, controlled in $z$, deterministic and continuous. 
  
  Especially, equation \eqref{eq:fbsde} has a unique solution $(Y,Z)$ such that $\|Z\|_\infty<\infty$.
\end{lemma}

\begin{proof}
  Using Theorem \ref{globalexistM} we know that there exists a unique weakly regular Markovian decoupling field $u$ on $I^M_{\mathrm{max}}$. This $u$ is strongly regular, controlled in $z$, deterministic and continuous. It remains to prove $I^M_{\mathrm{max}}=[0,T]$. Due to Lemma \ref{explosionM} it is sufficient to show the existence of a constant $C\in [t,\infty]$ such that $L_{u(t,\cdot ),x} \leq C < L_{\sigma, z}^{-1}$ for all $t\in I^M_{\mathrm{max}}$. In our case $L_{\sigma, z}^{-1}=\infty$, so we have to prove that the weak partial derivatives of $u$ with respect to $x^{(1)}$ and $x^{(2)}$ are both uniformly bounded.
  
  Fix $t\in I^M_{\mathrm{max}}$ and consider the corresponding FBSDE on $[t,T]$: First notice that the associated triplet $(X,Y,Z)$ depends on the initial value $x=(x^{(1)},x^{(2)})\in\mathbb{R}^2$, even in a weakly differentiable way with respect to the initial value $x$, according to the strong regularity of $u$. For more details about weak derivatives we refer to Chapter 2 of \cite{Fromm2015}, Section 2.1.2.
  
  Let us look at the matrix $\frac{\dx}{\dx x}X$. We observe that
  \begin{align*}
    & \frac{\dx}{\dx x^{(1)}}X^{(1)}_s=1,
    && \frac{\dx}{\dx x^{(1)}}X^{(2)}_s=\int_{t}^s 2 Z_r\frac{\dx}{\dx x^{(1)}} Z_r \dx r,\\
    & \frac{\dx}{\dx x^{(2)}}X^{(1)}_s=0,
    && \frac{\dx}{\dx x^{(2)}}X^{(2)}_s=1+\int_{t}^s 2 Z_r\frac{\dx}{\dx x^{(2)}} Z_r \dx r, \\
  \end{align*}
  a.s. for $s\in[t,T]$, for almost all $x=(x^{(1)},x^{(2)})\in\mathbb{R}^2$. In particular, the $2\times 2$-matrix $\frac{\dx}{\dx x}X_s$ is invertible if and only if $\frac{\dx}{\dx x^{(2)}}X^{(2)}_s$ ist not $0$. We will see later that it remains positive on the whole interval allowing us to apply the chain rule of Lemma \ref{chainruleambr} in order to write $\frac{\dx}{\dx x}u(s, X_s)\frac{\dx}{\dx x}X_s$.
  But let us first proceed by differentiating the backward equation in \eqref{eq:fbsde2} with respect to $ x^{(2)}$:
  \begin{equation*}
    \frac{\dx}{\dx x^{(2)}}Y_s=-\delta'(X^{(2)}_T)\frac{\dx}{\dx x^{(2)}}X^{(2)}_T-\int_{s}^{T}\frac{\dx}{\dx x^{(2)}}Z_r\dx W_r.
  \end{equation*}
  To be precise the above holds a.s. for every $s\in[t,T]$, for almost all $x=(x^{(1)},x^{(2)})\in\mathbb{R}^2$. 
  
  Now define a stopping time $\tau$ via
  \begin{equation*}
    \tau:=\inf \bigg\{ s \in [t,T]\, : \, \frac{\dx}{\dx x^{(2)}}X^{(2)}_s \leq 0\bigg \}\wedge T.
  \end{equation*}
  For $s\in [t, \tau)$ we have $\frac{\dx}{\dx x}u(s, X_s)\frac{\dx}{\dx x}X_s$ according to the chain rule of Lemma \ref{chainruleambr} and in particular $\frac{\dx}{\dx x^{(2)}}u(s, X^{(1)}_s, X^{(2)}_s)\frac{\dx}{\dx x^{(2)}}X^{(2)}_s=\frac{\dx}{\dx x^{(2)}}Y_s$. Let us set
  \begin{equation*}
    V_s:=\frac{\dx}{\dx x^{(2)}}u(s, X^{(1)}_s, X^{(2)}_s),\,s\in[t,T] \quad\text{ and }\quad \tilde{Z}_r:=\frac{\frac{\dx}{\dx x^{(2)}}Z_r}{\frac{\dx}{\dx x^{(2)}}X^{(2)}_r}\mathbf{1}_{\{r\in [t, \tau)\}}.
  \end{equation*}
  Then the dynamics of $\big(\frac{\dx}{\dx x^{(2)}}X^{(2)}_s\big)^{-1}$ can be expressed by
  \begin{equation}\label{eq:dynamicX}
    \left(\frac{\dx}{\dx x^{(2)}}X^{(2)}_{s\wedge \tilde\tau}\right)^{-1}=1-\int_t^{s\wedge \tilde\tau} 2 Z_r\tilde{Z}_r \left(\frac{\dx}{\dx x^{(2)}}X^{(2)}_r\right)^{-1} \dx r,
  \end{equation}
  for an arbitrary stopping time $\tilde\tau< \tau$ with values in $[t,T]$. We also have $\frac{\dx}{\dx x^{(2)}}Y_s=V_{s}\frac{\dx}{\dx x^{(2)}}X^{(2)}_s$ and thus
  \begin{equation*}
    V_{s}= \frac{\dx}{\dx x^{(2)}}Y_s\bigg(\frac{\dx}{\dx x^{(2)}}X^{(2)}_s\bigg)^{-1}, \quad s \in [t, \tau).
  \end{equation*}
  Applying It\^o's formula and using the dynamics of $\frac{\dx}{\dx x^{(2)}}Y$ and $\frac{\dx}{\dx x^{(2)}}X^{(2)}$  we easily obtain an equation describing the dynamics of $V_{s\wedge \tilde\tau}$:
  \begin{align}\label{Vdyn}
    V_{s\wedge \tilde\tau} &= V_t + \int_t^{s\wedge \tilde\tau } -2 Z_r \tilde{Z}_r \bigg(\frac{\dx}{\dx x^{(2)}}X^{(2)}_r\bigg)^{-1} \frac{\dx}{\dx x^{(2)}} Y_r \dx r
               + \int_t^{s\wedge \tilde\tau} \frac{\dx}{\dx x^{(2)}}Z_r \bigg(\frac{\dx}{\dx x^{(2)}}X^{(2)}_r\bigg)^{-1}  \dx W_r \nonumber\\
        &= V_t + \int_t^{s\wedge \tilde\tau} (- 2Z_r V_r) \tilde{Z}_r \dx r + \int_t^{s\wedge \tilde\tau}\tilde{Z}_r\dx W_r
  \end{align}
  for any stopping time $\tilde\tau< \tau$ with values in $[t,T]$.
  
  Note that,  since $V$ and $(-2ZV)$ are bounded processes, $\tilde{Z}\mathbf{1}_{[\cdot\leq\tilde\tau]}$ is in $BMO(\mathbb{P})$ according to Theorem \ref{BSDEBMO} with a  $BMO(\mathbb{P})$-norm which does not depend on $\tilde\tau< \tau$, and so in particular $\mathbb{E}[\int_t^\tau \vert 2Z_r\tilde Z_r \vert^2 \dx r ] < \infty$. From \eqref{eq:dynamicX} we can actually deduce that $\tau = T$ must hold almost surely. Indeed, \eqref{eq:dynamicX} implies that
  \begin{equation*}
    \left(\frac{\dx}{\dx x^{(2)}}X^{(2)}_{s\wedge \tilde\tau}\right)^{-1}= \exp\bigg(-\int_t^{s\wedge \tilde\tau} 2 Z_r\tilde{Z}_r  \dx r\bigg)
  \end{equation*}
  or equivalently
  \begin{equation*}
    \frac{\dx}{\dx x^{(2)}}X^{(2)}_{s\wedge \tilde\tau}= \exp\bigg(\int_t^{s\wedge \tilde\tau} 2 Z_r\tilde{Z}_r  \dx r\bigg)
  \end{equation*}
  for all stopping times $\tilde\tau< \tau$ with values in $[t,T]$. Using continuity of $s\mapsto \frac{\dx}{\dx x^{(2)}}X^{(2)}_{s}$ we obtain
  \begin{equation*}
    \frac{\dx}{\dx x^{(2)}}X^{(2)}_{\tau}= \exp\bigg(\int_t^{\tau} 2 Z_r\tilde{Z}_r  \dx r\bigg)>0,
  \end{equation*}
  which gives us $\tau = T$ a.s. because $\{\tau<T\}\subset \big\{ \frac{\dx}{\dx x^{(2)}}X^{(2)}_{\tau}=0 \big\}$, due to continuity of $\frac{\dx}{\dx x^{(2)}}X^{(2)}$.\\
  Hence, we have that $\frac{\dx}{\dx x^{(2)}}X^{(2)}$ is positive on $[t,T]$ and therefore $\frac{\dx}{\dx x}X$ is invertible on $[t,T]$.
 
  Setting $\tilde{W}_s:=W_s-\int_t^s2Z_rV_r\dx r$, $s\in[t,T]$, we can reformulate \eqref{Vdyn} to
  \begin{equation*}
    V_s=V_t+\int_t^s\tilde{Z}_r\dx \tilde{W}_r.
  \end{equation*}
  This means that $V_s$ can be viewed as the conditional expectation of
  \begin{equation*}
    V_T=\frac{\dx}{\dx x^{(2)}}u(T, X^{(1)}_T, X^{(2)}_T)=-\delta'(X^{(2)}_T)
  \end{equation*}
  with respect to $\mathcal{F}_s$ and some probability measure, which turns $\tilde{W}$ into a Brownian motion on $[t,T]$. Note here that $2Z_rV_r$ is bounded on $[t,T]$ because $\vert\vert Z \vert\vert_{\infty} < \infty$. Hence, we conclude that $V_t$ and therefore $\frac{\dx}{\dx x^{(2)}}u(t, x^{(1)}, x^{(2)})$ is bounded by $\|\delta'\|_\infty$ for almost all $x=(x^{(1)},x^{(2)})\in\mathbb{R}^2$. This value is independent of $t$. 
  
  Secondly, we have to bound $\frac{\dx}{\dx x^{(1)}}u(t, x^{(1)}, x^{(2)})$. To this end we differentiate the equations in \eqref{eq:fbsde2} with respect to $x^{(1)}$:
  \begin{align*}
    & \frac{\dx}{\dx x^{(1)}}X^{(1)}_s=1,\quad  \frac{\dx}{\dx x^{(1)}}X^{(2)}_s=\int_{t}^s 2 Z_r\frac{\dx}{\dx x^{(1)}} Z_r \dx r, \allowdisplaybreaks \\
    & \frac{\dx}{\dx x^{(1)}}Y_s= g'(X^{(1)}_T) - \delta'(X^{(2)}_T)\frac{\dx}{\dx x^{(1)}}X^{(2)}_T-\int_{s}^{T}\frac{\dx}{\dx x^{(1)}}Z_r\dx W_r, \\
    & \frac{\dx}{\dx x^{(1)}}u(s, X^{(1)}_s, X^{(2)}_s) + \frac{\dx}{\dx x^{(2)}}u(s, X^{(1)}_s, X^{(2)}_s)\frac{\dx}{\dx x^{(1)}}X^{(2)}_s=\frac{\dx}{\dx x^{(1)}}Y_s,
  \end{align*}
  and define
  \begin{equation*}
    U_s:=\frac{\dx}{\dx x^{(1)}}u(s, X^{(1)}_s, X^{(2)}_s) \quad \text{and}\quad
    \check{Z}_r:=\frac{\dx}{\dx x^{(1)}}Z_r-\tilde{Z}_r\frac{\dx}{\dx x^{(1)}}X^{(2)}_r.
  \end{equation*}
  Note that
  \begin{align*}
    \frac{\dx}{\dx x^{(1)}}X^{(2)}_s=\int_{t}^s 2 Z_r\left(\check{Z}_r+\tilde{Z}_r\frac{\dx}{\dx x^{(1)}}X^{(2)}_r\right)\dx r\quad \text{and} \quad
    U_s=\frac{\dx}{\dx x^{(1)}}Y_s-V_s\frac{\dx}{\dx x^{(1)}}X^{(2)}_s,
  \end{align*}
  which allows us to deduce the dynamics of $U$ from the dynamics of $\frac{\dx}{\dx x^{(1)}}Y$, $\frac{\dx}{\dx x^{(1)}}X^{(2)}$ and $V$ using It\^o formula:
  \begin{align}\label{Udyn}
    U_s =& U_t + \int_t^s 1 \dx \bigg( \frac{\dx}{\dx x^{(1)}}Y_r\bigg) - \int_t^s V_r \dx \bigg( \frac{\dx}{\dx x^{(1)}}X^{(2)}_r\bigg)
           - \int_t^s \frac{\dx}{\dx x^{(1)}}X^{(2)}_r  \dx V_r \nonumber\\\
        =& U_t + \int_t^s \frac{\dx}{\dx x^{(1)}}Z_r\dx W_r- 2\int_t^s V_r Z_r\left(\check{Z}_r+\tilde{Z}_r\frac{\dx}{\dx x^{(1)}}X^{(2)}_r\right) \dx r \nonumber\\\
         & -\int_t^s \frac{\dx}{\dx x^{(1)}}X^{(2)}_r\left( - 2Z_r V_r\tilde{Z}_r \dx r+\tilde{Z}_r\dx W_r\right)\nonumber\\
        =& U_t + \int_t^s (-2Z_rV_r\check{Z}_r)\dx r+\int_t^s\check{Z}_r\dx W_r
        = U_t + \int_t^s\check{Z}_r\dx \tilde W_r.
  \end{align}
  By the same argument as for the process $V$ we deduce that $U$ and therefore $\frac{\dx}{\dx x^{(1)}}u(t, x^{(1)}, x^{(2)})$ is bounded by $\|g'\|_\infty=L_g$ for almost all $(x^{(1)}, x^{(2)})$, where $L_g$ is the Lipschitz constant of $g$, i.e. the infimum of all Lipschitz constants. 
 
  This shows that $I^M_{\mathrm{max}}=[0,T]$. 
  
  Finally, Lemma \ref{uniqXYZM} shows that there is a unique solution $(X,Y,Z)$ to the FBSDE on $[0,T]$ for any initial value $(X_0^{(1)},X_0^{(2)})=(x^{(1)},x^{(2)})\in\mathbb{R}^2$ such that
  \begin{equation*}
    \sup_{s\in[0,T]}\mathbb{E}[|X_s|^2]+\sup_{s\in[0,T]}\mathbb{E}[|Y_s|^2]+\|Z\|_\infty<\infty,
  \end{equation*}
  which is equivalent to the simpler condition $\|Z\|_\infty<\infty$ as we claim.

  If $\|Z\|_\infty<\infty$, then according to the forward equation
  \begin{align*}
    \|X^{(2)}\|_\infty<\infty  \quad\text{and}\quad
    \sup_{s\in[0,T]}\mathbb{E}[|X_s|^2]=|x^{(1)}|^2+\sup_{s\in[0,T]}\mathbb{E}[|W_s|^2]=|x^{(1)}|^2+T<\infty,
  \end{align*}
  and according to the backward equation together with the Minkowski inequality
  \begin{align*}
    \bigg(\sup_{s\in[0,T]}\mathbb{E}[|Y_s|^2]\bigg)^{\frac{1}{2}}
    &= \left(\sup_{s\in[0,T]}\mathbb{E}\left[\left|\mathbb{E}\left[ g(X^{(1)}_T)- \delta (X^{(2)}_T)\bigg|\mathcal{F}_s\right]\right|^2\right]\right)^{\frac{1}{2}}  \\
    &\leq\left(\mathbb{E}\left[\left| g(X^{(1)}_T)- \delta (X^{(2)}_T)\right|^2\right]\right)^{\frac{1}{2}} \\
    &\leq|g(0)|+L_g\left(\mathbb{E}\left[\left| X^{(1)}_T\right|^2\right]\right)^{\frac{1}{2}}+|\delta(0)|+L_\delta\left(\mathbb{E}\left[\left| X^{(2)}_T\right|^2\right]\right)^{\frac{1}{2}}<\infty,
  \end{align*}
  where $L_g$ and $L_\delta$ are Lipschitz constants of $g$ and $\delta$, respectively.
\end{proof}

In the next lemma we investigate the properties of the control process $Z$ which was obtained in Lemma \ref{lem:ext}.

\begin{lemma}\label{lem:Zcontrol}
  Assume that $\delta$ and $g$ are Lipschitz continuous. Let $u$ be the unique weakly regular Markovian decoupling field associated to the problem \eqref{eq:fbsde2} on $[0,T]$ constructed in Lemma \ref{lem:ext}. Then for any $t\in[0,T)$ and initial condition $(X_t^{(1)},X_t^{(2)})=(x^{(1)},x^{(2)})\in\mathbb{R}^2$ the associated process $Z$ on $[t,T]$ satisfies $\|Z\|_\infty\leq L_{g}=\|g'\|_\infty$.
  
  Furthermore, if the weak derivative $\frac{\dx}{\dx x^{(1)}}u$ has a version which is continuous in the first two components $(s,x^{(1)})$ on $[t,T)\times\mathbb{R}^2$ then $Z_s(\omega)=\frac{\dx}{\dx x^{(1)}}u\big(s, X^{(1)}_s(\omega), X^{(2)}_s(\omega)\big)$ for almost all $(s,\omega)\in[t,T]\times\Omega$.
\end{lemma}

\begin{proof}
  We already know that $Z$ is bounded according to Lemma \ref{lem:ext}, but not in the form of the more explicit bound $\|Z\|_\infty\leq L_{g}$. 
  
  Notice that $\lim_{h\downarrow 0}\frac{1}{h}\int_{s}^{s+h} Z_r(\omega)\dx r=Z_{s}(\omega)$ for almost all $(\omega,s)\in\Omega\times[t,T)$ due to the fundamental theorem of Lebesgue integral calculus. 
  
  Now take some $s\in [t,T)$ such that  $\lim_{h\downarrow 0}\frac{1}{h}\int_{s}^{s+h} Z_r\dx r=Z_{s}$ almost surely. Almost all $s\in [t,T)$ have this property. Choose any $h>0$ such that $s+h<T$ and consider the expression $\frac{1}{h}\mathbb{E}[Y_{s+h}(W_{s+h}-W_s)|\mathcal{F}_s]$ for small $h>0$. On the one hand we can write using It\^o's formula
  \begin{equation*}
  Y_{s+h}(W_{s+h}-W_s)=\int_s^{s+h}Y_{r}\dx W_{r}+\int_s^{s+h}(W_{r}-W_s)Z_{r}\dx W_{r}+\int_s^{s+h}Z_{r}\dx r,
  \end{equation*}
  which leads to
  \begin{equation*}
    \frac{1}{h}\mathbb{E}[Y_{s+h}(W_{s+h}-W_s)|\mathcal{F}_s]=\frac{1}{h}\mathbb{E}\left[\int_s^{s+h}Z_{r}\dx r\bigg|\mathcal{F}_s\right]\to Z_s\quad\text{as}\quad h\to 0.
  \end{equation*}
  On the other hand we can use the decoupling condition to write
  \begin{align*}
    Y_{s+h}(W_{s+h}-W_s)=&u\left(s+h,X^{(1)}_{s+h},X^{(2)}_{s+h}\right)(W_{s+h}-W_s)\\
    =&u\left(s+h,X^{(1)}_{s+h},X^{(2)}_{s}\right)(W_{s+h}-W_s)\\
     &+\left(u\left(s+h,X^{(1)}_{s+h},X^{(2)}_{s+h}\right)-u\left(s+h,X^{(1)}_{s+h},X^{(2)}_{s}\right)\right)(W_{s+h}-W_s).
  \end{align*}
  After applying conditional expectations to both sides of the above equation we investigate the two summands on the right hand side separately. 
  
  \textsc{First summand:} Let us recall that $X^{(1)}_{s}$ and $X^{(2)}_{s}$ are $\mathcal{F}_s$-measurable, $X^{(1)}_{s+h}=X^{(1)}_{s}+(W_{s+h}-W_s)$, $W_{s+h}-W_s$ is independent of $\mathcal{F}_s$, and $u$ is deterministic. These properties imply
  \begin{align*}
    \mathbb{E}\bigg[u\bigg(s+h,X^{(1)}_{s+h},X^{(2)}_{s}\bigg)(W_{s+h}-W_s)&\Big|\mathcal{F}_s\bigg]
    =\int_{\mathbb{R}}u\bigg(s+h,X^{(1)}_{s}+z\sqrt{h},X^{(2)}_{s}\big)\frac{z\sqrt{h}}{\sqrt{2\pi}}e^{-\frac{1}{2}z^2}\dx z\\
    &=\int_{\mathbb{R}}\frac{\dx}{\dx x^{(1)}}u\bigg(s+h,X^{(1)}_{s}+z\sqrt{h},X^{(2)}_{s}\bigg)\frac{h}{\sqrt{2\pi}}e^{-\frac{1}{2}z^2}\dx z,
  \end{align*}
  which means
  \begin{equation*}
    \lim_{h\downarrow 0}\frac{1}{h}\mathbb{E}\left[u\left(s+h,X^{(1)}_{s+h},X^{(2)}_{s}\right)(W_{s+h}-W_s)\Big|\mathcal{F}_s\right]=\frac{\dx}{\dx x^{(1)}}u\left(s,X^{(1)}_{s},X^{(2)}_{s}\right),
  \end{equation*}
  if $\frac{\dx}{\dx x^{(1)}}u$ is continuous in the first two components on $[0,T)\times\mathbb{R}^2$. Here we use that $\frac{\dx}{\dx x^{(1)}}u$ is bounded by $\|g'\|_\infty$ according to the proof of Lemma \ref{lem:ext}. But even if $\frac{\dx}{\dx x^{(1)}}u$ is not continuous in the first two components, we can still at least control the value
  \begin{equation*}
    \left|\frac{1}{h}\mathbb{E}\left[u\left(s+h,X^{(1)}_{s+h},X^{(2)}_{s}\right)(W_{s+h}-W_s)\Big|\mathcal{F}_s\right]\right|
  \end{equation*}
  by $\|g'\|_\infty$.
  
  \textsc{Second summand:} For the second summand we recall that $u$ is Lipschitz continuous in the last component with Lipschitz constant $\|\delta'\|_\infty$ and $X^{(2)}_{s+h}=X^{(2)}_{s}+\int_s^{s+h}Z^2_r\dx r$. These properties allow us to estimate
  \begin{align*}
    \frac{1}{h}&\left|\mathbb{E}\left[\left(u\left(s+h,X^{(1)}_{s+h},X^{(2)}_{s+h}\right)-u\left(s+h,X^{(1)}_{s+h},X^{(2)}_{s}\right)\right)(W_{s+h}-W_s)\Big|\mathcal{F}_s\right]\right| \\
    &\leq\frac{1}{h}\mathbb{E}\left[\left|u\left(s+h,X^{(1)}_{s+h},X^{(2)}_{s+h}\right)-u\left(s+h,X^{(1)}_{s+h},X^{(2)}_{s}\right)\right|\cdot |W_{s+h}-W_s|\Big|\mathcal{F}_s\right] \\
    &\leq\frac{1}{h}\mathbb{E}\left[\|\delta'\|_\infty\left(\int_s^{s+h}Z^2_r\dx r\right)\cdot |W_{s+h}-W_s|\Big|\mathcal{F}_s\right]\leq\frac{1}{h}\|\delta'\|_\infty h \|Z\|^2_\infty\mathbb{E}[|W_{s+h}-W_s|],
  \end{align*}
  which clearly tends to $0$ as $h\to 0$.
  
  \textsc{Conclusion:} We have shown
  \begin{equation*}
    Z_s=\lim_{h\downarrow 0}\frac{1}{h}\mathbb{E}[Y_{s+h}(W_{s+h}-W_s)|\mathcal{F}_s]=\lim_{h\downarrow 0}\frac{1}{h}\mathbb{E}\left[u\left(s+h,X^{(1)}_{s+h},X^{(2)}_{s}\right)(W_{s+h}-W_s)\Big|\mathcal{F}_s\right],
  \end{equation*}
  which is identical with $\frac{\dx}{\dx x^{(1)}}u\big(s,X^{(1)}_{s},X^{(2)}_{s}\big)$ a.s. if $\frac{\dx}{\dx x^{(1)}}u$ is continuous in the first two components on $[0,T)\times\mathbb{R}^2$ and bounded by $\|g'\|_\infty$ otherwise.
\end{proof}

In order to formulate the weak solution of the Skorokhod embedding problem in the next theorem,  we use the notations of Section \ref{sec:SEP}. As before we assume that $\beta$ is bounded away from $0$. Under this condition $H^{-1}$ is well-defined and Lipschitz continuous. Therefore, $\delta=\hat{\delta}\circ H^{-1}$ is Lipschitz continuous if $\hat\delta$ is Lipschitz continuous, which is equivalent to $\alpha$ being bounded.

\begin{thm}\label{thm:weak solution}
  Suppose $g$ and $\delta$ are both Lipschitz continuous with Lipschitz constants $L_g$ and $L_\delta$. Then there exist a Brownian motion $B$, a random time $\tilde \tau\leq H^{-1}(L_g^2)$ and a constant $c\in\mathbb{R}$ such that $ c+\int_0^{\tilde\tau}\alpha_s\dx s+\int_0^{\tilde\tau}\beta_s\dx B_s$ has law $\nu$.
\end{thm}

\begin{proof}
  Using Lemma \ref{lem:ext} FBSDE \eqref{eq:fbsde} can be solved and according to Lemma \ref{lem:Zcontrol} the corresponding $Z$ is bounded by $L_g$. Due to Lemma \ref{lem:weak} there is a constant $c:=Y_0$, a Brownian motion $B$ and random time $\tilde\tau$ with the required properties.
  
  Moreover, $\tilde \tau= H^{-1}\big(\int_0^1 Z^2_s \dx s\big)$ is bounded by $H^{-1}(L^2_g)$ since $Z$ is bounded by $L_g$ and $H^{-1}$ is increasing.
\end{proof}

\begin{remark}
  It is a priori not clear that the random time $\tilde\tau$ is also a stopping time with respect to $\left(\mathcal{F}^B_s\right)_{s\in[0,\infty)}:=\big(\sigma(B_r,r\in[0,s])\big)_{s\in [0,\infty)}$ as also mentioned in Remark 1.2 in \cite{Ankirchner2008}. Therefore, we shall prove a sufficient criterion for this in terms of regularity properties of the Markovian decoupling field $u$.
\end{remark}

\begin{remark}
  The boundedness of the stopping time solving the SEP has not been investigated so frequently. However, very recently it gained attention in \cite{Ankirchner2011} and \cite{Ankirchner2015}. Especially, its economic interest comes from its applications in the context of game theory (see \cite{Seel2013}).
\end{remark}

\subsection{Strong solution}

This subsection is devoted to the {\bf fourth step} of our algorithm, i.e. to translate the results of the preceding section into a solution of the Skorokhod embedding problem in the strong sense.

Our main goal is to show that if $g$ and $\delta$ are sufficiently smooth, then $\tilde\tau$ and $B$ constructed so far have the property that $\tilde\tau$ is indeed a stopping time with respect to the filtration $\left(\mathcal{F}^B_s\right)_{s\in[0,\infty)}$ generated by the Brownian motion $B$, and thus a functional of the trajectories of $B$. The same functional applied to the trajectories of the original Brownian motion $W$ will then provide the strong solution. For this purpose, we assume that $g$ and $\delta$ are three times weakly differentiable with bounded derivatives. We also require that $g$ is non-decreasing and not constant. Our arguments shall be based on a deep analysis of regularity properties of the associated decoupling field $u$. In the whole subsection we denoted by $u$ the unique weakly regular Markovian decoupling field to the problem \eqref{eq:fbsde2} as constructed in Lemma \ref{lem:ext}, assume for convenience $T=1$ and use the notation as in Section \ref{sec:SEP}.

\begin{thm}\label{thm:mainresult}
  Assume that $\frac{\dx}{\dx x^{(1)}}u$ is $\mathbb{R}\backslash\{0\}$-valued on $[0,1)\times\mathbb{R}^2$ and Lipschitz continuous in the first two components on compact subsets of $[0,1)\times\mathbb{R}^2$. Then $\tilde\tau$ is a stopping time with respect to the filtration $(\mathcal{F}^B_\cdot)=(\mathcal{F}^B_s)_{s\in[0,\infty)}$.
\end{thm}

\begin{proof}
  We consider the system \eqref{eq:fbsde2} for $t=0$ and $x^{(1)}=x^{(2)}=0$. According to Lemma \ref{lem:Zcontrol} we can assume $Z=\frac{\dx}{\dx x^{(1)}}u\big(\cdot, X^{(1)}_\cdot, X^{(2)}_\cdot\big)$ and, thereby, we have
  \begin{equation*}
    X^{(2)}_s=\int_0^s Z^2_r\dx r=\int_0^s \left(\frac{\dx}{\dx x^{(1)}}u\left(r, X^{(1)}_r, X^{(2)}_r\right)\right)^2\dx r
  \end{equation*}
  for all $s\in[0,T]$. Hence, we can assume that $X^{(1)}$ starts in $0$, and is Lipschitz continuous and strictly increasing in $s$ due to positivity of $\big(\frac{\dx}{\dx x^{(1)}}u\big)^2$ on $[0,1)\times\mathbb{R}^2$. Therefore, for every $\omega\in\Omega$ the mapping $H^{-1}\big(X^{(2)}_\cdot(\omega)\big) \colon [0,1]\to [0,\infty)$ is Lipschitz continuous and strictly increasing in time and has a continuous and strictly increasing inverse function on the interval $\big[0,H^{-1}\big(X^{(2)}_1(\omega)\big)\big]=[0,\tilde\tau(\omega)]$.
  It is straightforward to see that this inverse is given by the process $\sigma$ from the proof of Lemma \ref{lem:weak}. Let us calculate the weak derivative of $\sigma$: Firstly, note $\big(H^{-1}\big)'(x)=(H'(H^{-1}(x)))^{-1}$ and also $H^{-1}(X^{(2)}_{\sigma_r}(\omega))=r$ or equivalently $X^{(2)}_{\sigma_r}(\omega)=H(r)$. So, we obtain
  \begin{align}\label{eq:sigmadym}
    \frac{\dx}{\dx r}\sigma_r
    = \frac{1}{\left(H^{-1}\right)'\left(X^{(2)}_{\sigma_r}\right)Z^2_{\sigma_r}}
    = \frac{H'(r)}{\left(\frac{\dx}{\dx x^{(1)}}u\left(\sigma_r, X^{(1)}_{\sigma_r}, X^{(2)}_{\sigma_r}\right)\right)^2}
    = \frac{\beta_r^2}{\left(\frac{\dx}{\dx x^{(1)}}u\right)^2\left(\sigma_r, W_{\sigma_r}, H(r)\right)}
  \end{align}
  on $\{\sigma_r<1\}$. Observe at this point that $ \{\sigma_r<1\} = \big\{r<H^{-1}\big(X^{(2)}_1\big)\big\}=\{r<\tilde{\tau}\}$. If we define $\sigma_r:=1$ for $r> \tilde\tau$, then $\sigma$ is still continuous and we have $\tilde\tau=\inf\left\{r\in[0,\infty)\,|\,\sigma_r\geq 1\right\}$. It is also straightforward to see $Z_{\sigma_r}=\frac{\dx}{\dx x^{(1)}}u\left(\sigma_r, W_{\sigma_r}, H(r)\right)$ for $r\in[0,\tilde\tau)$.
  
  Now, remember $B_r=\int_0^r\frac{1}{\beta_s}\dx Y_{\sigma_s}$ for $r\in[0,\tilde\tau]$ and also $Y_s-Y_0=\int_0^s Z_r\dx W_r$ for $s\in[0,1]$, so
  \begin{equation*}
    \int_0^r \frac{\beta_s}{Z_{\sigma_s}}\dx B_{s}=\int_0^r \frac{\beta_s}{Z_{\sigma_s}}\frac{1}{\beta_s}\dx Y_{\sigma_s}=\int_0^r \frac{1}{Z_{\sigma_s}}Z_{\sigma_s}\dx W_{\sigma_s}=W_{\sigma_r}.
  \end{equation*}
  So, if we define $\Sigma_r:=W_{\sigma_r}$, we have the dynamics
  \begin{equation*}
    \Sigma_r=\int_0^r \frac{\beta_s}{\frac{\dx}{\dx x^{(1)}}u\left(\sigma_s, \Sigma_s, H(s)\right)}\dx B_{s},
  \end{equation*}
  for $r\in[0,\tilde\tau)$. Hence, to sum up $\sigma$ and $\Sigma$ fulfill on $[0,\tilde\tau)$ the dynamics
  \begin{align*}
    \sigma_r=\int_0^r \frac{\beta_s^2}{\left(\frac{\dx}{\dx x^{(1)}}u\right)^2\left(\sigma_s, \Sigma_s, H(s)\right)}\dx s\quad\text{and}\quad
    \Sigma_r=\int_0^r \frac{\beta_s}{\frac{\dx}{\dx x^{(1)}}u\left(\sigma_s, \Sigma_s, H(s)\right)}\dx B_{s},
  \end{align*}
  where $r\in[0,\tilde\tau)$. Note that this dynamical system is locally Lipschitz continuous in $(\sigma,\Sigma)$.
  
  Moreover, for any $K_1,K_2>0$ and $K_3\in (0,1)$ define a bounded random variable $\tau_{K_1,K_2,K_3}$ via
  \begin{equation*}
    \tau_{K_1,K_2,K_3}:=K_1\wedge \inf\left\{r\in[0,\infty)\,|\,|\Sigma_r|\geq K_2\right\}\wedge\inf\left\{r\in[0,\infty)\,|\,\sigma_r\geq K_3\right\}.
  \end{equation*}
  Note that $\sigma$ and $\Sigma$ both remain bounded on $[0,\tau_{K_1,K_2,K_3}]$. Therefore, on $[0,\tau_{K_1,K_2,K_3}]$ the pair $(\sigma,\Sigma)$ coincides with the unique solution $(\sigma^{K_1,K_2,K_3},\Sigma^{K_1,K_2,K_3})$ to a Lipschitz problem, which is automatically progressively measurable w.r.t. the filtration $(\mathcal{F}^B_\cdot)$. Note that
  \begin{equation*}
    \tau_{K_1,K_2,K_3}=K_1\wedge \inf\left\{r\in[0,\infty\,\big|\,|\Sigma^{K_1,K_2,K_3}_r|\geq K_2\right\}\wedge \inf\left\{r\in[0,\infty)\,\big|\,\sigma^{K_1,K_2,K_3}_r\geq K_3\right\},
  \end{equation*}
  which is clearly a stopping time w.r.t. $(\mathcal{F}^B_\cdot)$. Furthermore, due to continuity of $\Sigma$ and $\sigma$ we Observe that
  \begin{equation*}
    \tilde\tau=\sup_{K_3\in(0,1),K_1,K_2>0}\tau_{K_1,K_2,K_3},
  \end{equation*}
  which makes it a stopping time with respect to $(\mathcal{F}^B_\cdot)$.
\end{proof}

In order to deduce sufficient conditions for Theorem \ref{thm:mainresult} to hold, we need to investigate higher order derivatives of $u$. For this purpose we consider the following system:
\begin{align}\label{eq:fbsde3}
  X^{(1)}_s &= x^{(1)}+ \int_{t}^s 1 \dx W_r, \qquad \qquad X^{(2)}_s = x^{(2)} +  \int_{t}^s  \left(Z^{(0)}_r\right)^2 \dx r, \notag\allowdisplaybreaks \\
  Y^{(0)}_s &= g(X^{(1)}_T)- \delta (X^{(2)}_T)-\int_{s}^{T}Z^{(0)}_r\dx W_r,  &u^{(0)}(s,X^{(1)}_s,X^{(2)}_s) &= Y^{(0)}_s, \notag \\
  Y^{(1)}_s &= g'(X^{(1)}_T)-\int_{s}^T Z^{(1)}_r \dx W_r-\int_{s}^{T}\left(-2Z^{(0)}_rY^{(2)}_r\right)Z^{(1)}_r\dx r,  &u^{(1)}(s,X^{(1)}_s,X^{(2)}_s) &= Y^{(1)}_s, \notag \\
  Y^{(2)}_s &= -\delta'(X^{(2)}_T)-\int_{s}^T Z^{(2)}_r \dx W_r-\int_{s}^{T}\left(-2Z^{(0)}_rY^{(2)}_r\right)Z^{(2)}_r\dx r,  &u^{(2)}(s,X^{(1)}_s,X^{(2)}_s) &= Y^{(2)}_s.
\end{align}

\begin{lemma}\label{lem:fbsde3}
  Assume that $g$, $\delta$, $g'$ and $\delta'$ are Lipschitz continuous. For the above problem \eqref{eq:fbsde3} we have $I^M_{\mathrm{max}}=[0,T]$. Furthermore, we obtain
  \begin{equation*}
    u^{(0)}=u,\quad u^{(1)}=\frac{\dx}{\dx x^{(1)}}u\quad \textrm{and}\quad u^{(2)}=\frac{\dx}{\dx x^{(2)}}u,\quad a.e.,
  \end{equation*}
  where $u$ is the unique weakly regular Markovian decoupling field to the problem \eqref{eq:fbsde2}. 
  
  In particular, $u$ is twice weakly differentiable w.r.t. $x$ with uniformly bounded derivatives.
\end{lemma}

\begin{proof}
  The proof is in parts akin to the proof of Lemma \ref{lem:ext} and we will seek to keep these parts short. 
  
  Let $u^{(i)}$, $i=0,1,2$, be the unique weakly regular Markovian decoupling field on $I^M_{\mathrm{max}}$. We can assume $u^{(i)}$ to be continuous functions on $I^M_{\mathrm{max}}\times\mathbb{R}^2$ (cf. Theorem \ref{globalexistM}). Let $t\in I^M_{\mathrm{max}}$. For an arbitrary initial condition $x\in\mathbb{R}^2$ we consider the corresponding processes $X^{(1)}$, $X^{(2)}$, $Y^{(0)}$, $Y^{(1)}$, $Y^{(2)}$, $Z^{(0)}$, $Z^{(1)}$ and $Z^{(2)}$ on $[t,T]$. Note that $X^{(1)}, X^{(2)}, Y^{(0)}, Z^{(0)}$ solve the FBSDE \eqref{eq:fbsde2}, which implies that they coincide with the processes $X^{(1)}, X^{(2)}, Y, Z$ from \eqref{eq:fbsde2} if we assume
  \begin{equation*}
    \sum_{i=1}^2\sup_{s\in[t,T]}\mathbb{E}_{0,\infty}[|X^{(i)}_s|^2]+\sup_{s\in[t,T]}\mathbb{E}_{0,\infty}[|Y_s|^2]+\|Z\|_\infty+\sum_{i=0}^2\sup_{s\in[t,T]}\mathbb{E}_{0,\infty}[|Y^{(i)}_s|^2]+\sum_{i=0}^2\|Z^{i}\|_\infty<\infty,
  \end{equation*}
  according to Lemma \ref{uniqXYZM}. This condition is fulfilled due to strong regularity and the fact that we work with Markovian decoupling fields.
  
  Now, $Y^{(0)}=Y$ implies $u(t,x)=u^{(0)}(t,x)$ for all $t\in I^M_{\mathrm{max}}, x\in\mathbb{R}^2$, where $I^M_{\mathrm{max}}$ is the maximal interval for the problem given by \eqref{eq:fbsde3}.
  We now claim that $Y^{(1)}$ and $Y^{(2)}$ are bounded processes: Using the backward equation we have
  \begin{equation*}
    Y^{(2)}_s = \mathbb{E}_s\left[-\delta'(X^{(2)}_T)\right]-\mathbb{E}_s\left[\int_{s}^{T}\left(-2Z^{(0)}_rY^{(2)}_r\right)Z^{(2)}_r\dx r\right]
  \end{equation*}
  and, therefore,
  \begin{equation*}
    |Y^{(2)}_s| \leq \|\delta'\|_\infty+\int_{s}^{T}2\|Z^{(0)}\|_\infty\|Z^{(2)}\|_\infty\mathbb{E}_s\left[\left|Y^{(2)}_r\right|\right]\dx r,
  \end{equation*}
  for $s\in[t,T]$, which using Gronwall's lemma implies
  \begin{equation*}
    |Y^{(2)}_s|=\mathbb{E}_s\left[\left|Y^{(2)}_s\right|\right]\leq \|\delta'\|_\infty\exp\left(2T\|Z^{(0)}\|_\infty\|Z^{(2)}\|_\infty\right).
  \end{equation*}
  This in turn automatically implies boundedness of $Y^{(1)}$ according to its dynamics. Furthermore, $Y^{(1)},Z^{(1)}$ and $Y^{(2)},Z^{(2)}$ satisfy the BSDE which is also fulfilled by the processes $U,\check{Z}$ and $V,\tilde{Z}$ from the proof of Lemma \ref{lem:ext} (see \eqref{Vdyn} and \eqref{Udyn}) and so in particular
  \begin{align*}
    Y^{(2)}_s-V_s
    &=0-\int_{s}^T \left(Z^{(2)}_r -\tilde{Z}_r \right)\dx W_r-\int_{s}^{T}\left(-2Z^{(0)}_r\right)\left(Y^{(2)}_rZ^{(2)}_r-V_r\tilde{Z}_r\right)\dx r \\
    &=0-\int_{s}^T \left(Z^{(2)}_r -\tilde{Z}_r \right)\dx W_r-\int_{s}^{T}\left(-2Z^{(0)}_r\right)\left(\left(Y^{(2)}_r V_r\right)Z^{(2)}_r+V_r\left(Z^{(2)}_r-\tilde{Z}_r\right)\right)\dx r.
  \end{align*}
  Using the boundedness of $Z^{(0)}$, $Z^{(2)}$ and $V$ this implies using Lemma \ref{lindimcont} that  $Y^{(2)}-V$ is $0$ almost everywhere. Therefore, after setting $\tilde{W}_s:=W_s-\int_t^s2Z^{(0)}_rV_r\dx r$, $s\in[t,T]$ we get from the above equation $\int_{s}^T \big(Z^{(2)}_r -\tilde{Z}_r \big)\dx \tilde{W}_r=0$ a.s. for $s\in[t,T]$. Since $\tilde{W}$ is a Brownian motion under some probability measure equivalent to $\mathbb{P}$ we also have $Z^{(2)}-\tilde{Z}=0$ a.e.
  
  Similarly, one shows that $Y^{(1)}$ and $U$ as well as  $Z^{(1)}$ and $\check{Z}$ coincide so
  \begin{equation*}
    Y^{(1)}=U,\quad Y^{(2)}=V,\quad  Z^{(1)}=\check{Z}\quad\textrm{and}\quad Z^{(2)}=\tilde{Z}\qquad\textrm{ a.e.}
  \end{equation*}
  Now, remember $U_s=\frac{\dx}{\dx x^{(1)}}u(s,X^{(1)}_s,X^{(2)}_s)$. Together with $u^{(1)}(s,X^{(1)}_s,X^{(2)}_s) = Y^{(1)}_s$ and $Y^{(1)}=U$ this yields $u^{(1)}(t,\cdot)=\frac{\dx}{\dx x^{(1)}}u(t,\cdot)$ and, therefore, $u^{(1)}=\frac{\dx}{\dx x^{(1)}}u$ a.e. on $I^M_{\mathrm{max}}$. Similarly, we get $u^{(2)}=\frac{\dx}{\dx x^{(2)}}u$.
  Further, note that $u^{(1)}=\frac{\dx}{\dx x^{(1)}}u$ is continuous. This makes Lemma \ref{lem:Zcontrol} applicable, so
  \begin{equation}\label{zzuy}
    Z^{(0)}=Z=U=Y^{(1)}\,\textrm{ a.e.}
  \end{equation}
  Thereby $Y^{(1)}$ and $Y^{(2)}$ satisfy the following dynamics:
  \begin{align}
    Y^{(1)}_s &= g'(X^{(1)}_T)-\int_{s}^T Z^{(1)}_r \dx W_r-\int_{s}^{T}\left(-2Y^{(1)}_rY^{(2)}_r\right)Z^{(1)}_r\dx r, \label{y1u} \\
    Y^{(2)}_s &= -\delta'(X^{(2)}_T)-\int_{s}^T Z^{(2)}_r \dx W_r-\int_{s}^{T}\left(-2Y^{(1)}_rY^{(2)}_r\right)Z^{(2)}_r\dx r, \quad s\in[t,T]\label{y2u},
  \end{align}
  which implies using the chain rule of Lemma \ref{chainruleambr}:
  \begin{align*}
    \frac{\dx}{\dx x^{(i)}}Y^{(1)}_s =& g''(X^{(1)}_T)\frac{\dx}{\dx x^{(i)}}X^{(1)}_T-\int_{s}^T \frac{\dx}{\dx x^{(i)}}Z^{(1)}_r \dx W_r\\
    &-\int_{s}^{T}(-2)\left(\left(\frac{\dx}{\dx x^{(i)}}Y^{(1)}_rY^{(2)}_r+Y^{(1)}_r\frac{\dx}{\dx x^{(i)}}Y^{(2)}_r\right)Z^{(1)}_r+Y^{(1)}_rY^{(2)}_r\frac{\dx}{\dx x^{(i)}}Z^{(1)}_r\right)\dx r,
  \end{align*}
  and
  \begin{align*}
    \frac{\dx}{\dx x^{(i)}}Y^{(2)}_s = &-\delta''(X^{(2)}_T)\frac{\dx}{\dx x^{(i)}}X^{(2)}_T-\int_{s}^T \frac{\dx}{\dx x^{(i)}}Z^{(2)}_r \dx W_r\\
    &-\int_{s}^{T}(-2)\left(\left(\frac{\dx}{\dx x^{(i)}}Y^{(1)}_rY^{(2)}_r+Y^{(1)}_r\frac{\dx}{\dx x^{(i)}}Y^{(2)}_r\right)Z^{(2)}_r+Y^{(1)}_rY^{(2)}_r\frac{\dx}{\dx x^{(i)}}Z^{(2)}_r\right)\dx r,
  \end{align*}
  for $i=1,2$. Let us recall some statements about the forward process obtained in the proof of Lemma \ref{lem:ext}:
  \begin{equation*}
    \frac{\dx}{\dx x^{(2)}}X^{(2)}>0,\quad \frac{\dx}{\dx x^{(1)}}X^{(1)}=1, \quad \frac{\dx}{\dx x^{(2)}}X^{(1)}=0, \quad\textrm{a.e.},
  \end{equation*}
  and
  \begin{align}
    \label{Xdyn1}&\left(\frac{\dx}{\dx x^{(2)}}X^{(2)}_s\right)^{-1}=1-\int_t^s2Y^{(1)}_rZ^{(2)}_r \left(\frac{\dx}{\dx x^{(2)}}X^{(2)}_r\right)^{-1} \dx r,\\
    \label{Xdyn2}&\frac{\dx}{\dx x^{(1)}}X^{(2)}_s=\int_{t}^s 2 Y^{(1)}_r\left(Z^{(1)}_r+Z^{(2)}_r\frac{\dx}{\dx x^{(1)}}X^{(2)}_r\right)\dx r.
  \end{align}
  Using the chain rule of Lemma \ref{chainruleambr} and the decoupling condition, we have
  \begin{align*}
    \frac{\dx}{\dx x^{(1)}}Y^{(i)}_s&=\frac{\dx}{\dx x^{(1)}}u^{(i)}(s,X^{(1)}_s,X^{(2)}_s)+\frac{\dx}{\dx x^{(2)}}u^{(i)}(s,X^{(1)}_s,X^{(2)}_s)\frac{\dx}{\dx x^{(1)}}X^{(2)}_s, \\
    \frac{\dx}{\dx x^{(2)}}Y^{(i)}_s&=\frac{\dx}{\dx x^{(2)}}u^{(i)}(s,X^{(1)}_s,X^{(2)}_s)\frac{\dx}{\dx x^{(2)}}X^{(2)}_s,\quad i=1,2.
  \end{align*}
  Let us set
  \begin{align}
    Y^{(12)}_s &:= \frac{\dx}{\dx x^{(2)}}u^{(1)}(s,X^{(1)}_s,X^{(2)}_s)=\left(\frac{\dx}{\dx x^{(2)}}Y^{(1)}_s\right)\left(\frac{\dx}{\dx x^{(2)}}X^{(2)}_s\right)^{-1},\label{y12} \\
    Y^{(22)}_s &:= \frac{\dx}{\dx x^{(2)}}u^{(2)}(s,X^{(1)}_s,X^{(2)}_s)=\left(\frac{\dx}{\dx x^{(2)}}Y^{(2)}_s\right)\left(\frac{\dx}{\dx x^{(2)}}X^{(2)}_s\right)^{-1},\notag \allowdisplaybreaks\\
    Y^{(11)}_s &:= \frac{\dx}{\dx x^{(1)}}u^{(1)}(s,X^{(1)}_s,X^{(2)}_s)=\frac{\dx}{\dx x^{(1)}}Y^{(1)}_s-Y^{(12)}_s\frac{\dx}{\dx x^{(1)}}X^{(2)}_s, \label{y11}\\
    Y^{(21)}_s &:= \frac{\dx}{\dx x^{(1)}}u^{(2)}(s,X^{(1)}_s,X^{(2)}_s)=\frac{\dx}{\dx x^{(1)}}Y^{(2)}_s-Y^{(22)}_s\frac{\dx}{\dx x^{(1)}}X^{(2)}_s.\notag
  \end{align}
  We can apply the It\^o formula to deduce dynamics of $Y^{(12)}$ and $Y^{(11)}$ from dynamics of $\frac{\dx}{\dx x^{(2)}}Y^{(1)}$, $\big(\frac{\dx}{\dx x^{(2)}}X^{(2)}\big)^{-1}$, $\frac{\dx}{\dx x^{(1)}}Y^{(1)}$ and $\frac{\dx}{\dx x^{(1)}}X^{(2)}$: 
  
  Let us define $Z^{(12)}_s:=\big(\frac{\dx}{\dx x^{(2)}}Z^{(1)}_s\big)\big(\frac{\dx}{\dx x^{(2)}}X^{(2)}_s\big)^{-1}$, so we can write using \eqref{y12}
  \begin{align*}
    Y^{(12)}_s =&0-\int_{s}^T Z^{(12)}_r \dx W_r
     -\int_{s}^{T}\bigg\{(-2)\left(\bigg(\frac{\dx}{\dx x^{(2)}}Y^{(1)}_rY^{(2)}_r+Y^{(1)}_r\frac{\dx}{\dx x^{(2)}}Y^{(2)}_r\right)Z^{(1)}_r
    \\&+Y^{(1)}_rY^{(2)}_r\frac{\dx}{\dx  x^{(2)}}Z^{(1)}_r\bigg)\left(\frac{\dx}{\dx x^{(2)}}X^{(2)}_r\right)^{-1}
    -2\frac{\dx}{\dx x^{(2)}}Y^{(1)}_s  Y^{(1)}_rZ^{(2)}_r \left(\frac{\dx}{\dx x^{(2)}}X^{(2)}_r\right)^{-1}\bigg\} \dx r.
  \end{align*}
  Using the definitions of $Y^{(12)}$, $Y^{(22)}$ and $Z^{(12)}$ we can simplify this to
  \begin{align*}
    Y^{(12)}_s = &0-\int_{s}^T Z^{(12)}_r \dx W_r\\
    & -\int_{s}^{T}(-2)\left(\left(Y^{(12)}_rY^{(2)}_r+Y^{(1)}_rY^{(22)}_r\right)Z^{(1)}_r+Y^{(1)}_rY^{(2)}_rZ^{(12)}_r+Y^{(12)}_rY^{(1)}_rZ^{(2)}_r\right) \dx r.
  \end{align*}
  Let us now define $Z^{(11)}_s:=\frac{\dx}{\dx x^{(1)}}Z^{(1)}_s-Z^{(12)}_s\frac{\dx}{\dx x^{(1)}}X^{(2)}_s$, so we can write using \eqref{y11}
  \begin{align*}
    Y^{(11)}_s
    =& g''(X^{(1)}_T)-\int_{s}^T Z^{(11)}_r \dx W_r \\
     & -\int_{s}^{T}\bigg\{(-2)\left(\left(\frac{\dx}{\dx x^{(1)}}Y^{(1)}_rY^{(2)}_r+Y^{(1)}_r\frac{\dx}{\dx x^{(1)}}Y^{(2)}_r\right)Z^{(1)}_r+Y^{(1)}_rY^{(2)}_r\frac{\dx}{\dx x^{(1)}}Z^{(1)}_r\right) \\
     &-(-2)\left(\left(Y^{(12)}_rY^{(2)}_r+Y^{(1)}_rY^{(22)}_r\right)Z^{(1)}_r+Y^{(1)}_rY^{(2)}_rZ^{(12)}_r+Y^{(12)}_rY^{(1)}_rZ^{(2)}_r\right)\frac{\dx}{\dx x^{(1)}}X^{(2)}_r \\
     &-Y^{(12)}_r \cdot 2\cdot Y^{(1)}_r\left(Z^{(1)}_r+Z^{(2)}_r\frac{\dx}{\dx x^{(1)}}X^{(2)}_r\right)\bigg\} \dx r.
  \end{align*}
  This can be simplified using \eqref{y11} to
  \begin{align*}
    Y^{(11)}_s
    =& g''(X^{(1)}_T)-\int_{s}^T Z^{(11)}_r \dx W_r\\
     & -\int_{s}^{T}\bigg\{(-2)\bigg(\bigg(Y^{(11)}_rY^{(2)}_r+Y^{(1)}_r\frac{\dx}{\dx x^{(1)}}Y^{(2)}_r\bigg)Z^{(1)}_r+Y^{(1)}_rY^{(2)}_r \frac{\dx}{\dx x^{(1)}}Z^{(1)}_r\bigg) \\
     &-(-2)\bigg(Y^{(1)}_r Y^{(22)}_rZ^{(1)}_r+Y^{(1)}_rY^{(2)}_r Z^{(12)}_r+Y^{(12)}_rY^{(1)}_rZ^{(2)}_r\bigg)\frac{\dx}{\dx x^{(1)}}X^{(2)}_r \\
     &-Y^{(12)}_r \cdot 2\cdot Y^{(1)}_r\bigg(Z^{(1)}_r+Z^{(2)}_r\frac{\dx}{\dx x^{(1)}}X^{(2)}_r\bigg)\bigg\} \dx r.
  \end{align*}
  Similarly, we merged further terms using the structure of $Y^{(21)}$ and $Z^{(11)}$ to get
  \begin{align*}
    Y^{(11)}_s =&g''(X^{(1)}_T)-\int_{s}^T Z^{(11)}_r \dx W_r\\
    &-\int_{s}^{T}\bigg\{(-2)\bigg(\bigg(Y^{(11)}_rY^{(2)}_r+Y^{(1)}_rY^{(21)}_r\bigg)Z^{(1)}_r+Y^{(1)}_rY^{(2)}_rZ^{(11)}_r\bigg) \\
    &-(-2)\bigg(Y^{(12)}_rY^{(1)}_rZ^{(2)}_r\bigg)\frac{\dx}{\dx x^{(1)}}X^{(2)}_r-Y^{(12)}_r \cdot 2\cdot Y^{(1)}_r\bigg(Z^{(1)}_r+Z^{(2)}_r\frac{\dx}{\dx x^{(1)}}X^{(2)}_r\bigg)\bigg\} \dx r\\
    =&g''(X^{(1)}_T)-\int_{s}^T Z^{(11)}_r \dx W_r\\
    &-\int_{s}^{T}(-2)\left(\left(Y^{(11)}_rY^{(2)}_r+Y^{(1)}_rY^{(21)}_r\right)Z^{(1)}_r+Y^{(1)}_rY^{(2)}_rZ^{(11)}_r+Y^{(12)}_r Y^{(1)}_rZ^{(1)}_r\right) \dx r.
  \end{align*}
  Analogously to $Y^{(12)}$ we can deduce dynamics of  $Y^{(22)}$:
  \begin{align*}
    Y^{(22)}_s =&-\delta''(X^{(2)}_T)-\int_{s}^T Z^{(22)}_r \dx W_r\\
    &-\int_{s}^{T}(-2)\left(\left(Y^{(12)}_rY^{(2)}_r+Y^{(1)}_rY^{(22)}_r\right)Z^{(2)}_r+Y^{(1)}_rY^{(2)}_rZ^{(22)}_r+Y^{(22)}_rY^{(1)}_rZ^{(2)}_r\right) \dx r.
  \end{align*}
  From here we can, analogously to $ Y^{(11)}$, deduce the dynamics of $Y^{(21)}$:
  \begin{align*}
    Y^{(21)}_s =&0-\int_{s}^T Z^{(21)}_r \dx W_r\\
    &-\int_{s}^{T}(-2)\left(\left(Y^{(11)}_rY^{(2)}_r+Y^{(1)}_rY^{(21)}_r\right)Z^{(2)}_r+Y^{(1)}_rY^{(2)}_rZ^{(21)}_r+Y^{(22)}_r Y^{(1)}_rZ^{(1)}_r\right) \dx r.
  \end{align*}
  And so we have finally obtained the complete dynamics of the $4$-dimensional process $(Y^{(ij)})$, $i,j=1,2$, which are clearly linear in it. Furthermore, remember:
  \begin{itemize}
    \item $Y^{(1)}$, $Y^{(2)}$ are uniformly bounded independently of $(t,x)$ due to the decoupling condition, $u^{(i)}=\frac{\dx}{\dx x^{(i)}}u$, $i=1,2$ and Lemma \ref{lem:ext},
    \item $Z^{(1)}$, $Z^{(2)}$ are $BMO(\mathbb{P})$ processes with uniformly bounded $BMO(\mathbb{P})$-norms independently of $(t,x)$ due to \eqref{y1u}, \eqref{y2u} and Theorem \ref{BSDEBMO},
    \item $(Y^{(ij)})$, $i,j=1,2$ are bounded according to their definition (with a bound which may depend on $t,x$ at this point),
    \item $(Z^{(ij)})$, $i,j=1,2$ are in $BMO(\mathbb{P})$ according to Theorem \ref{BSDEBMO},
    \item $(Y_T^{(ij)})_{i,j=1,2}$ is uniformly bounded by $\|g''\|_\infty+\|\delta''\|_\infty<\infty$.
  \end{itemize}
  Therefore, Lemma \ref{lindimcont} is applicable and $(Y^{(ij}))_{i,j=1,2}$ is uniformly bounded, independently of $(t,x)$. In particular, $Y^{(ij)}_t=\frac{\dx}{\dx x^{(j)}}u^{(i)}(t,x)$, $i,j=1,2$, can be controlled independently of $t\in I^M_{\mathrm{max}}$, $x\in\mathbb{R}^2$, while $\frac{\dx}{\dx x^{(j)}}u^{(0)}(t,x)$, $j=1,2$, has the same property as we already know. This shows $I^M_{\mathrm{max}}=[0,T]$ using Lemma \ref{explosionM}.
\end{proof}

\begin{lemma}\label{lem:Zcontrol2}
  Let $g$, $\delta$, $g'$, $\delta'$ be Lipschitz continuous functions. Let $\left(u^{(i)}\right)_{i=0,1,2}$ be the unique weakly regular Markovian decoupling field to problem \eqref{eq:fbsde3} constructed in Theorem \ref{lem:fbsde3}.
  
  Suppose that $\frac{\dx}{\dx x^{(1)}}u^{(i)}$, $i=0,1,2$, has a version which is continuous in the first two components $(s,x^{(1)})$ on $[t,T)\times\mathbb{R}^2$ for some $t\in[0,T)$. Then for any initial condition $(X_t^{(1)},X_t^{(2)})=(x^{(1)},x^{(2)})=x\in\mathbb{R}^2$ the associated processes $Z^{(i)}$, $i=0,1,2$, on $[t,T]$ satisfy
  \begin{equation*}
  Z^{(i)}_s(\omega)=\frac{\dx}{\dx x^{(1)}}u^{(i)}\left(s, X^{(1)}_s(\omega), X^{(2)}_s(\omega)\right), \quad i=0,1,2,
  \end{equation*}
  for almost all $(s,\omega)\in[t,T]\times\Omega$. 
  
  Furthermore, in this case the processes
  \begin{equation*}
    \frac{\dx}{\dx x^{(1)}}X^{(2)},\quad \frac{\dx}{\dx x^{(2)}}X^{(2)}\quad\textrm{and}\quad\bigg(\frac{\dx}{\dx x^{(2)}}X^{(2)}\bigg)^{-1} \textrm{ on } [t,T],
  \end{equation*}
  can be bounded uniformly, i.e. independently of $(t,x)$.
\end{lemma}

\begin{proof}
  The first part of the proof works analogously to the proof of Lemma \ref{lem:Zcontrol}. So we keep our arguments short. For $i=0,1,2$ we consider $\frac{1}{h}\mathbb{E}[Y^{(i)}_{s+h}(W_{s+h}-W_s)|\mathcal{F}_s]$ for small $h>0$. As in the proof of Lemma \ref{lem:Zcontrol}, we use It\^o's formula applied to \eqref{eq:fbsde3} to obtain
  \begin{align*}
    Y^{(i)}_{s+h}(W_{s+h}-W_s)
    =&\int_s^{s+h}Y^{(i)}_{r}\dx W_{r}+\int_s^{s+h}(W_{r}-W_s)Z^{(i)}_{r}\dx W_{r}\\
     &+\int_s^{s+h}(W_{r}-W_s)\left(-2Z^{(0)}_rY^{(2)}_r\right)Z^{(i)}_{r}\dx r+\int_s^{s+h}Z^{(i)}_{r}\dx r,
  \end{align*}
  and also
  \begin{equation*}
    Y^{(0)}_{s+h}(W_{s+h}-W_s)=\int_s^{s+h}Y^{(0)}_{r}\dx W_{r}+\int_s^{s+h}(W_{r}-W_s)Z^{(0)}_{r}\dx W_{r}+\int_s^{s+h}Z^{(0)}_{r}\dx r,
  \end{equation*}
  which leads to
  \begin{equation*}
    \frac{1}{h}\mathbb{E}[Y^{(0)}_{s+h}(W_{s+h}-W_s)|\mathcal{F}_s]=\frac{1}{h}\mathbb{E}\left[\int_s^{s+h}Z^{(0)}_{r}\dx r\bigg|\mathcal{F}_s\right]\to Z^{(0)}_s\quad\textrm{for}\quad h\to 0,
  \end{equation*}
  and
  \begin{equation*}
    \frac{1}{h}\mathbb{E}[Y^{(i)}_{s+h}(W_{s+h}-W_s)|\mathcal{F}_s]=\frac{1}{h}\mathbb{E}\left[\int_s^{s+h}Z^{(i)}_{r}\left(1+(W_{r}-W_s)\left(-2Z^{(0)}_rY^{(2)}_r\right)\right)\dx r\bigg|\mathcal{F}_s\right]\to Z^{(i)}_s
  \end{equation*}
  as $h\to 0$ for $i=1,2$. The arguments are valid for almost all $s\in[t,T]$.
  
  On the other hand we can use the decoupling condition to rewrite
  \begin{align*}
    Y^{(i)}_{s+h}(W_{s+h}-W_s) = &u^{(i)}\left(s+h,X^{(1)}_{s+h},X^{(2)}_{s}\right)(W_{s+h}-W_s)\\
      &+\left(u^{(i)}\left(s+h,X^{(1)}_{s+h},X^{(2)}_{s+h}\right)-u^{(i)}\left(s+h,X^{(1)}_{s+h},X^{(2)}_{s}\right)\right)(W_{s+h}-W_s).
  \end{align*}
  Let us deal separately with the two summands. For the first one recall that $X^{(1)}_{s}$ and $X^{(2)}_{s}$ are $\mathcal{F}_s$-measurable, $X^{(1)}_{s+h}=X^{(1)}_{s}+(W_{s+h}-W_s)$, $W_{s+h}-W_s$ is independent of $\mathcal{F}_s$, and $u$ is deterministic, i.e. is assumed to be a function of $\left(s,x^{(1)},x^{(2)}\right)\in[0,T]\times\mathbb{R}^2$. A combination of these properties leads to
  \begin{equation*}
    \lim_{h\downarrow 0}\frac{1}{h}\mathbb{E}\left[u^{(i)}\left(s+h,X^{(1)}_{s+h},X^{(2)}_{s}\right)(W_{s+h}-W_s)\Big|\mathcal{F}_s\right]=\frac{\dx}{\dx x^{(1)}}u^{(i)}\left(s,X^{(1)}_{s},X^{(2)}_{s}\right),
  \end{equation*}
  if $\frac{\dx}{\dx x^{(1)}}u^{(i)}$ is continuous in the first two components on $[t,T)\times\mathbb{R}^2$, where we use that $\frac{\dx}{\dx x^{(1)}}u^{(i)}$ is bounded.
  
  For the second summand recall that $u^{(i)}$ is also Lipschitz continuous in the last component with some Lipschitz constant $L$ and $X^{(2)}_{s+h}=X^{(2)}_{s}+\int_s^{s+h}\big( Z_r^{(0)}\big)^2\dx r$. These properties allow us to estimate
  \begin{align*}
    \frac{1}{h}&\left|\mathbb{E}\left[\left(u^{(i)}\left(s+h,X^{(1)}_{s+h},X^{(2)}_{s+h}\right)-u^{(i)}\left(s+h,X^{(1)}_{s+h},X^{(2)}_{s}\right)\right)(W_{s+h}-W_s)\Big|\mathcal{F}_s\right]\right| \\
    &\leq\frac{1}{h}\mathbb{E}\left[L\cdot\left(\int_s^{s+h}\big( Z_r^{(0)}\big)^2\dx r\right)\cdot |W_{s+h}-W_s|\Big|\mathcal{F}_s\right]\leq\frac{1}{h}L\cdot h \|Z^{(0)}\|^2_\infty\mathbb{E}[|W_{s+h}-W_s|],
  \end{align*}
  which tends to $0$ as $h\to 0$.
  
  Therefore, we can conclude
  \begin{equation*}
    Z^{(i)}_s=\lim_{h\downarrow 0}\frac{1}{h}\mathbb{E}[Y^{(i)}_{s+h}(W_{s+h}-W_s)|\mathcal{F}_s]=\frac{\dx}{\dx x^{(1)}}u^{(i)}\left(s,X^{(1)}_{s},X^{(2)}_{s}\right)
  \end{equation*}
  if $\frac{\dx}{\dx x^{(1)}}u^{(i)}$ is continuous in the first two components on $[t,T)\times\mathbb{R}^2$, for $i=0,1,2$.
  
  Now recall \eqref{Xdyn1} and \eqref{Xdyn2} from the proof of Theorem \ref{lem:fbsde3}:
  \begin{align*}
    &\bigg(\frac{\dx}{\dx x^{(2)}}X^{(2)}_s\bigg)^{-1}=1-\int_t^s2Y^{(1)}_rZ^{(2)}_r \bigg(\frac{\dx}{\dx x^{(2)}}X^{(2)}_r\bigg)^{-1} \dx r,\\
    &\frac{\dx}{\dx x^{(1)}}X^{(2)}_s=\int_{t}^s 2 Y^{(1)}_r\bigg(Z^{(1)}_r+Z^{(2)}_r\frac{\dx}{\dx x^{(1)}}X^{(2)}_r\bigg)\dx r,
  \end{align*}
  a.s. for $s\in[t,T]$. The first equation implies
  \begin{equation*}
    \bigg(\frac{\dx}{\dx x^{(2)}}X^{(2)}_s\bigg)^{-1}=\exp\bigg(-\int_t^s2Y^{(1)}_rZ^{(2)}_r \dx r\bigg).
  \end{equation*}
  Using $Z^{(2)}=\frac{\dx}{\dx x^{(1)}}u^{(2)}(\cdot,X^{(1)}_{\cdot},X^{(2)}_{\cdot})$, $Y^{(1)}=Z^{(0)}=\frac{\dx}{\dx x^{(1)}}u^{(0)}(\cdot,X^{(1)}_{\cdot},X^{(2)}_{\cdot})$ (see \eqref{zzuy} in the proof of Theorem \ref{lem:fbsde3}) and uniform boundedness of $\frac{\dx}{\dx x^{(1)}}u^{(i)}$ for $i=0,1,2$ we see that this implies uniform boundedness of $\big(\frac{\dx}{\dx x^{(2)}}X^{(2)}_s\big)^{-1}$ and its inverse $\frac{\dx}{\dx x^{(2)}}X^{(2)}_s$.
  
  Furthermore, we have 
  \begin{equation*}
    \left|\frac{\dx}{\dx x^{(1)}}X^{(2)}_s\right|\leq 2T\|Y^{(1)}Z^{(1)}\|_\infty+\int_{t}^s 2 \|Y^{(1)}Z^{(2)}\|_\infty\left|\frac{\dx}{\dx x^{(1)}}X^{(2)}_r\right|\dx r.
  \end{equation*}
  By Gronwall's lemma together with uniform boundedness of $Z^{(1)}=\frac{\dx}{\dx x^{(1)}}u^{(1)}(\cdot,X^{(1)}_{\cdot},X^{(2)}_{\cdot})$,  $Z^{(2)}$ and $Y^{(1)}$ implies the uniform boundedness of $\frac{\dx}{\dx x^{(1)}}X^{(2)}$.
\end{proof}

For the subsequent results we employ the following notation:
\begin{itemize}
  \item For a real number $H>0$ let $\chi_H:\mathbb{R}\rightarrow\mathbb{R}$ be defined via $\chi_H(x):=(-H)\vee \left(x\wedge H\right)$ for $x\in\mathbb{R}$. In particular, $\chi_H$ is bounded, Lipschitz continuous and coincides with the identity function on the interval $[-H,H]$.
  \item For real numbers $y^{(ij)}$ and $y^{(i)}$ we denote by $y^{(ij)\wedge H}$ and $y^{(i)\wedge H}$ the values $\chi_H(y^{(ij)})$ and $\chi_H(y^{(i)})$ for $i,j=1,2$.
\end{itemize}

To prove sufficiently regularity properties of the decoupling field $u$, we need to consider for $H>0$ the following even higher dimensional system of equations:
\begin{align*}
  X^{(1)}_s  = x^{(1)}+ \int_{t}^s 1 \dx W_r,\quad
  X^{(2)}_s  = x^{(2)} +\int_{t}^s \left(Z^{(0)}_r\right)^2 \dx r
\end{align*}
with backward equations
\begin{align*}
  Y^{(0)}_s  &= g(X^{(1)}_T)- \delta (X^{(2)}_T)-\int_{s}^{T}Z^{(0)}_r\dx W_r,  & u^{(0)}(s,X^{(1)}_s,X^{(2)}_s) &= Y^{(0)}_s,& \notag \\
  Y^{(1)}_s  &= g'(X^{(1)}_T)-\int_{s}^T Z^{(1)}_r \dx W_r-\int_{s}^{T}\left(-2Z^{(0)}_rY^{(2)}_r\right)Z^{(1)}_r\dx r, &  u^{(1)}(s,X^{(1)}_s,X^{(2)}_s) &= Y^{(1)}_s,& \notag \\
  Y^{(2)}_s  &= -\delta'(X^{(2)}_T)-\int_{s}^T Z^{(2)}_r \dx W_r-\int_{s}^{T}\left(-2Z^{(0)}_rY^{(2)}_r\right)Z^{(2)}_r\dx r,& u^{(2)}(s,X^{(1)}_s,X^{(2)}_s) &= Y^{(2)}_s,
\end{align*}
and
\begin{align*}
  Y^{(11)}_s =&  g''(X^{(1)}_T)- \int_{s}^T Z^{(11)}_r \dx W_r - \int_{s}^{T}(-2)\bigg\{\left(Y^{(11)\wedge H}_rY^{(2)\wedge H}_r+Y^{(1)\wedge H}_rY^{(21)\wedge H}_r\right)Z^{(1)}_r& \notag \\
  &+Y^{(1)\wedge H}_rY^{(2)\wedge H}_rZ^{(11)}_r+Y^{(12)\wedge H}_rY^{(1)\wedge H}_rZ^{(1)}_r\bigg\}\dx r, \notag \allowdisplaybreaks\\
  Y^{(12)}_s =&  0- \int_{s}^T Z^{(12)}_r \dx W_r-\int_{s}^{T}(-2)\bigg\{\left(Y^{(12)\wedge H}_rY^{(2)\wedge H}_r+Y^{(1)\wedge H}_rY^{(22)\wedge H}_r\right)Z^{(1)}_r \notag \\
  & +Y^{(1)\wedge H}_rY^{(2)\wedge H}_rZ^{(12)}_r+Y^{(12)\wedge H}_rY^{(1)\wedge H}_rZ^{(2)}_r\bigg\}\dx r, \notag \allowdisplaybreaks\\
  Y^{(21)}_s =& 0- \int_{s}^T Z^{(21)}_r \dx W_r-\int_{s}^{T}(-2)\bigg\{\left(Y^{(11)\wedge H}_rY^{(2)\wedge H}_r+Y^{(1)\wedge H}_rY^{(21)\wedge H}_r\right)Z^{(2)}_r \notag \\
  &+Y^{(1)\wedge H}_rY^{(2)\wedge H}_rZ^{(21)}_r+Y^{(22)\wedge H}_rY^{(1)\wedge H}_rZ^{(1)}_r\bigg\}\dx r, \notag \allowdisplaybreaks\\
  Y^{(22)}_s =&-\delta''(X^{(2)}_T)- \int_{s}^T Z^{(22)}_r \dx W_r-\int_{s}^{T}(-2)\bigg\{\left(Y^{(12)\wedge H}_rY^{(2)\wedge H}_r+Y^{(1)\wedge H}_rY^{(22)\wedge H}_r\right)Z^{(2)}_r \notag \\
  &+Y^{(1)\wedge H}_rY^{(2)\wedge H}_rZ^{(22)}_r+Y^{(22)\wedge H}_rY^{(1)\wedge H}_rZ^{(2)}_r\bigg\}\dx r, \notag
\end{align*}
and with the decoupling conditions
\begin{align}\label{eq:fbsde4}
  u^{(11)}(s,X^{(1)}_s,X^{(2)}_s) &= Y^{(11)}_s,&   u^{(12)}(s,X^{(1)}_s,X^{(2)}_s) &= Y^{(12)}_s, & \notag \\
  u^{(21)}(s,X^{(1)}_s,X^{(2)}_s) &= Y^{(21)}_s,&   u^{(22)}(s,X^{(1)}_s,X^{(2)}_s) &= Y^{(22)}_s. &
\end{align}
With \eqref{eq:fbsde4} we will always refer to all the above equations.

\begin{lemma}\label{lem:fbsde4}
  Let $g$, $\delta$, $g'$, $\delta'$, $g''$, $\delta''$ be Lipschitz continuous functions. For sufficiently large $H>0$ the above problem \eqref{eq:fbsde4} satisfies $I^M_{\mathrm{max}}=[0,T]$ and in addition
  \begin{align*}
    &u^{(0)}=u,\quad u^{(1)}=\frac{\dx}{\dx x^{(1)}}u, \quad u^{(2)}=\frac{\dx}{\dx x^{(2)}}u, \quad u^{(11)}=\frac{\dx^2}{\left(\dx x^{(1)}\right)^2}u, \\
    &u^{(12)}=\frac{\dx}{\dx x^{(2)}}\frac{\dx}{\dx x^{(1)}}u, \quad u^{(21)}=\frac{\dx}{\dx x^{(1)}}\frac{\dx}{\dx x^{(2)}}u,\quad u^{(22)}=\frac{\dx^2}{\left(\dx x^{(2)}\right)^2}u,\quad a.e.,
  \end{align*}
  where $u$ is the unique weakly regular Markovian decoupling field to the problem \eqref{eq:fbsde2}. In particular, $u$ is three times weakly differentiable w.r.t. $x$ with uniformly bounded derivatives.
\end{lemma}

\begin{proof}
  The proof is in parts akin to the proof of Lemma \ref{lem:ext} and we will again seek to keep these parts short. 
  
  Assume $I^M_{\mathrm{max}}=(t^M_{\mathrm{min}},T]$ and $t\in I^M_{\mathrm{max}}$. Let $u^{(i)}$ and $u^{(jk)}$, $i=0,1,2$, $j,k=1,2$, be the associated weakly regular decoupling field on $I^M_{\mathrm{max}}$. We want to control $\frac{\dx}{\dx x}u^{(i)}u(t,\cdot)$, $\frac{\dx}{\dx x}u^{(jk)}(t,\cdot)$, $i=0,1,2$ and $j,k=1,2$, independently of $t$ to create a contradiction with respect to Lemma \ref{explosionM}. 
  
  For this purpose we consider the first three components of the decoupling field. Since $\left(u^{(i)}\right)_{i=0,1,2}$ is clearly a weakly regular Markovian decoupling field to the problem \eqref{eq:fbsde3} the mappings $\left(u^{(i)}\right)_{i=0,1,2}$ in \eqref{eq:fbsde3} and in \eqref{eq:fbsde4} are identical according to Theorem \ref{uniqM} and the processes $X^{(1)}$, $X^{(2)}$, $Y^{(i)}$, $Z^{(i)}$, $i=0,1,2$, in \eqref{eq:fbsde3} must coincide with the identically denoted processes in \eqref{eq:fbsde4} according to strong regularity. This is true for every $t\in I^M_{\mathrm{max}}$ and initial condition $x\in\mathbb{R}^2$. Hence, we can apply Theorem \ref{lem:fbsde3} and get
  \begin{equation*}
    u^{(0)}=u,\quad u^{(1)}=\frac{\dx}{\dx x^{(1)}}u, \quad u^{(2)}=\frac{\dx}{\dx x^{(2)}}u\quad \textrm{ on}\quad I^M_{\mathrm{max}}.
  \end{equation*}
  In particular, the last two functions are uniformly bounded.
  
  Furthermore, we saw in the proof of Theorem \ref{lem:fbsde3} that  $Y^{(1)}$ and $Y^{(2)}$ are uniformly bounded independently of $(t,x)$ and $Z^{(1)}$ and $Z^{(2)}$ are $BMO(\mathbb{P})$ processes with uniformly bounded $BMO(\mathbb{P})$-norms independently of $(t,x)$. Especially, $Y^{(i)\wedge H}=Y^{(i)}$ for $i=1,2$ if we make $H$ large enough. We will make this assumption from now on. 
  
  The processes $Y^{(jk)}$, $j,k=1,2$, satisfy
  \begin{align*}
    Y^{(jk)}_s =&Y^{(jk)}_T- \int_{s}^T Z^{(jk)}_r \dx W_r \\
     &-\int_{s}^{T}\bigg(\sum_{l_1,l_2,l_3,l_4=1,2} \alpha^{(jk)}_{l_1,l_2,l_3,l_4}Y^{(l_1)}_rZ^{(l_2)}_rY^{(l_3l_4)\wedge H}_r+Y^{(1)}_rY^{(2)}_rZ^{(jk)}_r\bigg)\dx r,
  \end{align*}
  where  $\alpha^{(jk)}_{l_1,l_2,l_3,l_4}$ is always either $0$ or $-2$. Since due to the structure of the terminal condition $Y^{(jk)}_T$ are uniformly bounded, we can apply Lemma \ref{lindimcont} to obtain uniform boundedness of $Y^{(jk)}$ as processes on $[t,T]$ independently of $(t,x)$.
  
  In particular, $Y^{(jk)\wedge H}=Y^{(jk)}$ for $j,k=1,2$ if we make $H$ large enough. We will make this assumption from now on. 
  
  This implies that the processes $Y^{(jk)}$, $j,k=1,2$, coincide with the identically denoted processes in the proof of Theorem \ref{lem:fbsde3} since they satisfy the same stochastic differential equations with the same terminal condition and we can apply Lemma \ref{lindimcont} to the difference of these four-dimensional processes obtaining that this difference must vanish. This implies however that $Y^{(jk)}_t=\frac{\dx}{\dx x^{(k)}}u^{(j)}\left(t, x^{(1)}, x^{(2)}\right)$ for almost all $(x^{(1)}, x^{(2)})$. So we obtain $u^{(jk)}=\frac{\dx}{\dx x^{(k)}}u^{(j)}$, $j,k=1,2$ a.e and these functions are uniformly bounded according to Theorem \ref{lem:fbsde3}. 
  
  According to Remark \ref{deterpluscontin}, the functions $\frac{\dx}{\dx x^{(1)}}u=u^{(1)}$ and $\frac{\dx}{\dx x^{(1)}}u^{(i)}=u^{(i1)}$, $i=1,2$, are continuous on $[t,T]\times\mathbb{R}^2$ and we can apply Lemma \ref{lem:Zcontrol2} to get $Z^{(i)}=\frac{\dx}{\dx x^{(1)}}u^{(i)}\big(\cdot, X^{(1)}_\cdot, X^{(2)}_\cdot\big)$ for $i=0,1,2$. Hence, $Z^{(i)}$ are uniformly bounded for $i=0,1,2$.
  
  Let us now analyze higher order derivatives $\frac{\dx}{\dx x^{(i)}}u^{(jk)}$ for $i,j,k=1,2$. As usual this is done by investigating equations characterizing the dynamics of $\frac{\dx}{\dx x^{(i)}}Y^{(jk)}$ for $i,j,k=1,2$. Using strong regularity we obtain
  \begin{align*}
    \frac{\dx}{\dx x^{(i)}}Y^{(jk)}_s = &\frac{\dx}{\dx x^{(i)}}Y^{(jk)}_T- \int_{s}^T \frac{\dx}{\dx x^{(i)}}Z^{(jk)}_r \dx W_r \\
    &-\int_{s}^{T}\bigg(G^{(jk)}_r+\sum_{l_1,l_2,l_3,l_4=1,2}\alpha^{(jk)}_{l_1,l_2,l_3,l_4}H^{(jk),l_1,l_2,l_3,l_4}_r\bigg)\dx r,
  \end{align*}
  where
  \begin{align*}
    & H^{i,(jk),l_1,l_2,l_3,l_4}_r=\frac{\dx}{\dx x^{(i)}}Y^{(l_1)}_rZ^{(l_2)}_rY^{(l_3l_4)}_r+ Y^{(l_1)}_r\frac{\dx}{\dx x^{(i)}}Z^{(l_2)}_rY^{(l_3l_4)}_r+Y^{(l_1)}_rZ^{(l_2)}_r\frac{\dx}{\dx x^{(i)}}Y^{(l_3l_4)}_r,\\
    &G^{i,(jk)}_r=\frac{\dx}{\dx x^{(i)}}Y^{(1)}_rY^{(2)}_rZ^{(jk)}_r+Y^{(1)}_r\frac{\dx}{\dx x^{(i)}}Y^{(2)}_rZ^{(jk)}_r+Y^{(1)}_rY^{(2)}_r\frac{\dx}{\dx x^{(i)}}Z^{(jk)}_r.
  \end{align*}
  This already implies that $\frac{\dx}{\dx x^{(i)}}Y^{(jk)}$, $i,j,k=1,2$, is uniformly bounded according to Lemma \ref{lindimcont}. The lemma is applicable since
  \begin{itemize}
    \item $\frac{\dx}{\dx x^{(i)}}Y^{(jk)}_T$ is either $0$ or has the structure $g^{(3)}(X^{(1)}_T)\frac{\dx}{\dx x^{(i)}}X^{(1)}_T$ or $-\delta^{(3)}(X^{(2)}_T)\frac{\dx}{\dx x^{(i)}}X^{(2)}_T$ which is uniformly bounded due to the Lipschitz continuity of $g'',\delta''$ and Lemma \ref{lem:Zcontrol2},
    \item $\frac{\dx}{\dx x^{(i)}}Y^{(l)}_r=\frac{\dx}{\dx x^{(1)}}u^{(l)}(r,X^{(1)}_r, X^{(2)}_r)\frac{\dx}{\dx x^{(i)}}X^{(1)}_r+\frac{\dx}{\dx x^{(2)}}u^{(l)}(r,X^{(1)}_r, X^{(2)}_r)\frac{\dx}{\dx x^{(i)}}X^{(2)}_r$ is also uniformly bounded according to Theorem \ref{lem:fbsde3} and  Lemma \ref{lem:Zcontrol2},
    \item $\frac{\dx}{\dx x^{(i)}}Y^{(jk)}_r=\frac{\dx}{\dx x^{(1)}}u^{(jk)}(r,X^{(1)}_r, X^{(2)}_r)\frac{\dx}{\dx x^{(i)}}X^{(1)}_r+\frac{\dx}{\dx x^{(2)}}u^{(jk)}(r,X^{(1)}_r, X^{(2)}_r)\frac{\dx}{\dx x^{(i)}}X^{(2)}_r$ is a bounded processes on $[t,T]$ according to Lemma \ref{lem:Zcontrol2} (but not necessarily uniformly in $t$ at this point),
    \item $\frac{\dx}{\dx x^{(i)}}Z^{(l)}_r=\frac{\dx}{\dx x^{(i)}}u^{(l)}\left(r, X^{(1)}_r, X^{(2)}_r\right)=\frac{\dx}{\dx x^{(i)}}Y^{(l1)}_r$ for all $l=1,2$,
    \item $Y^{(l_1l_2)}$, $Y^{(l)}$, $Z^{(l)}$ are always uniformly bounded as was already mentioned,
    \item $Z^{(l_1l_2)}$ are $BMO(\mathbb{P})$-processes with uniformly bounded $BMO(\mathbb{P})$ - norms according to the equations describing $Y^{(l_1l_2)}$ and Theorem \ref{BSDEBMO}.
  \end{itemize}
  Let $j,k\in\{1,2\}$. As a consequence of the decoupling condition together with the chain rule of Lemma \ref{chainruleambr}, we have
  \begin{align*}
    &\frac{\dx}{\dx x^{(1)}}Y^{(jk)}_r=\frac{\dx}{\dx x^{(1)}}u^{(jk)}(r,X^{(1)}_r, X^{(2)}_r)+\frac{\dx}{\dx x^{(2)}}u^{(jk)}(r,X^{(1)}_r, X^{(2)}_r)\frac{\dx}{\dx x^{(1)}}X^{(2)}_r,\\
    &\frac{\dx}{\dx x^{(2)}}Y^{(jk)}_r=\frac{\dx}{\dx x^{(2)}}u^{(jk)}(r,X^{(1)}_r, X^{(2)}_r)\frac{\dx}{\dx x^{(2)}}X^{(2)}_r.
  \end{align*}
  Using the boundedness of $\big(\frac{\dx}{\dx x^{(2)}}X^{(2)}\big)^{-1}$, the last equation implies that $\frac{\dx}{\dx x^{(2)}}u^{(jk)}(t,x^{(1)},x^{(2)})$ is bounded for almost all $x^{(1)},x^{(2)}$ by a uniform constant. Now the first equation together with uniform boundedness of $\frac{\dx}{\dx x^{(1)}}X^{(2)}_r$ and $\frac{\dx}{\dx x^{(1)}}Y^{(jk)}_r$ implies uniform boundedness of $\frac{\dx}{\dx x^{(1)}}u^{(jk)}$ as well.
  
  Considering Lemma \ref{explosionM} we have a contradiction and the proof is complete.
\end{proof}

\begin{thm}\label{thm:strong}
  Let $T=1$ and $g$, $\delta$, $g'$, $\delta'$, $g''$, $\delta''$ be Lipschitz continuous functions. Suppose additionally that $g$ is increasing and not constant.
  Then the Markovian decoupling field $u$ from Lemma \ref{lem:ext} fulfills the requirements of Theorem \ref{thm:mainresult}.
\end{thm}

\begin{proof}
  Let $\left(u^{(0)},u^{(1)},u^{(2)},u^{(11)},u^{(12)},u^{(21)},u^{(22)}\right)$ be the unique Markovian decoupling field to the problem \eqref{eq:fbsde4} on $[0,T]$. We have $u^{(0)}=u$, $u^{(1)}=\frac{\dx}{\dx x^{(1)}}u$, etc. according to Theorem \ref{lem:fbsde4}. 
  
  Let us show that $\frac{\dx}{\dx x^{(1)}}u$ is Lipschitz continuous in the first component. For this purpose we consider for a starting time $t\in [0,T]$ and initial condition $x\in\mathbb{R}^2$ the associated FBSDE \eqref{eq:fbsde4} on $[t,1]$. Recall that
  \begin{equation}\label{thisdec}
    Y^{(1)}_s=\frac{\dx}{\dx x^{(1)}}u(s, X^{(1)}_s, X^{(2)}_s),\quad s\in[t,1],
  \end{equation}
  satisfies
  \begin{equation}\label{thisy1}
    Y^{(1)}_s=Y^{(1)}_t + \int_t^s \left(-2Z^{(0)}_rY^{(2)}_r\right)Z^{(1)}_r\dx r+\int_t^sZ^{(1)}_r\dx W_r,\quad s\in[t,1],
  \end{equation}
  where
  \begin{itemize}
    \item $Z^{(0)}=\frac{\dx}{\dx x^{(1)}}u^{(0)}\left(\cdot, X^{(1)}_\cdot, X^{(2)}_\cdot\right)=Y^{(1)}$ a.e. according to Lemma \ref{lem:Zcontrol2}, which is applicable since $\left(\frac{\dx}{\dx x^{(1)}}u^{(i)}\right)_{i=1,2}=\left(u^{(i1)}\right)_{i=1,2}$ and $\frac{\dx}{\dx x^{(1)}}u^{(0)}=u^{(1)}$ are continuous on $[t,1]$ according to Remark \ref{deterpluscontin},
    \item $Z^{(0)}=Y^{(1)}$ and $Y^{(2)}$ are bounded by $\left\|\frac{\dx}{\dx x^{(1)}}u\right\|_\infty$ and $\left\|\frac{\dx}{\dx x^{(2)}}u\right\|_\infty$,
    \item $Z^{(1)}=\frac{\dx}{\dx x^{(1)}}u^{(1)}\big(\cdot, X^{(1)}_\cdot, X^{(2)}_\cdot\big)$ a.e. according to Lemma \ref{lem:Zcontrol2}, which is applicable as already mentioned. So $Z^{(1)}$ is bounded by $\left\|\frac{\dx}{\dx x^{(1)}}u^{(1)}\right\|_\infty$.
  \end{itemize}
  Let $s\in (t,1]$. Using the triangular inequality we obtain
  \begin{align*}
    \left|\frac{\dx}{\dx x^{(1)}}u(s,x)-\frac{\dx}{\dx x^{(1)}}u(t,x)\right|
    \leq &\left|\frac{\dx}{\dx x^{(1)}}u(s,x)-\mathbb{E}\left[\frac{\dx}{\dx x^{(1)}}u(s, X^{(1)}_s, X^{(2)}_s)\right]\right| \\
         &+\left|\mathbb{E}\left[\frac{\dx}{\dx x^{(1)}}u(s, X^{(1)}_s, X^{(2)}_s)\right]-\frac{\dx}{\dx x^{(1)}}u(t,x)\right|.
  \end{align*}
  Applying the triangular inequality for a second time together with \eqref{thisdec} we get
  \begin{align*}
    \bigg|\frac{\dx}{\dx x^{(1)}}&u(s,x)-\frac{\dx}{\dx x^{(1)}}u(t,x)\bigg|\\
    \leq &\left|\frac{\dx}{\dx x^{(1)}}u(s,x^{(1)}, x^{(2)})-\mathbb{E}\left[\frac{\dx}{\dx x^{(1)}}u(s, X^{(1)}_s, x^{(2)})\right]\right| \\
    & +\left|\mathbb{E}\left[\frac{\dx}{\dx x^{(1)}}u(s, X^{(1)}_s, x^{(2)})\right]-\mathbb{E}\left[\frac{\dx}{\dx x^{(1)}}u(s, X^{(1)}_s, X^{(2)}_s)\right]\right|+
    \left|\mathbb{E}\left[Y^{(1)}_s-Y^{(1)}_t\right]\right|.
  \end{align*}
  Let us now control the three summands on the right-hand-side separately.

  \textsc{First summand:} Let us define
  \begin{equation*}
    \varphi(z):=\frac{\dx}{\dx x^{(1)}}u(s,x^{(1)}, x^{(2)})-\frac{\dx}{\dx x^{(1)}}u(s, x^{(1)}+z, x^{(2)}),\quad z\in\mathbb{R},
  \end{equation*}
  and note:
  \begin{itemize}
    \item $\left|\frac{\dx}{\dx x^{(1)}}u(s,x^{(1)}, x^{(2)})-\mathbb{E}\left[\frac{\dx}{\dx x^{(1)}}u(s, X^{(1)}_s, x^{(2)})\right]\right|=
    \left|\int_{\mathbb{R}}\varphi(\sqrt{s-t}z)\frac{1}{\sqrt{2\pi}}e^{-\frac{1}{2}z^2}\dx z\right|$ as
    $X^{(1)}_s=x^{(1)}+W_s-W_t\sim\mathcal{N}\left(x^{(1)},s-t\right),$
    \item $\varphi$ is Lipschitz continuous with Lipschitz constant $L_{u^{(1)}}$, which is the Lipschitz constant of
    $\frac{\dx}{\dx x^{(1)}}u=u^{(1)}$ w.r.t. the last two components, and $\varphi(0)=0$,
    \item $\varphi'$ is Lipschitz continuous with Lipschitz constant $L_{u^{(11)}}$, which is the Lipschitz constant of $\frac{\dx^2}{\left(\dx x^{(1)}\right)^2}u=u^{(11)}$ w.r.t. the last two components.
  \end{itemize}
  And so using Lemma \ref{secdivcon} we obtain
  \begin{equation*}
    \left|\frac{\dx}{\dx x^{(1)}}u(s,x^{(1)}, x^{(2)})-\mathbb{E}\left[\frac{\dx}{\dx x^{(1)}}u(s, X^{(1)}_s, x^{(2)})\right]\right|\leq \frac{1}{2}(s-t)\cdot L_{u^{(11)}}.
  \end{equation*}
  
  \textsc{Second summand:} One sees that 
  \begin{align*}
    \bigg|\mathbb{E}\bigg[\frac{\dx}{\dx x^{(1)}}u(&s, X^{(1)}_s, x^{(2)})\bigg]-\mathbb{E}\bigg[\frac{\dx}{\dx x^{(1)}}u(s, X^{(1)}_s, X^{(2)}_s)\bigg]\bigg| \\
    &\leq\mathbb{E}\left[\left|\frac{\dx}{\dx x^{(1)}}u(s, X^{(1)}_s, x^{(2)})-\frac{\dx}{\dx x^{(1)}}u(s, X^{(1)}_s, X^{(2)}_s)\right|\right]\leq
    L_{u^{(1)}}\mathbb{E}\left[\left|X^{(2)}_s-x^{(2)}\right|\right],
  \end{align*}
  while
  \begin{equation*}
    \left|X^{(2)}_s-x^{(2)}\right|=\left|\int_t^s\left(Z^{(0)}_r\right)^2\dx r\right|\leq (s-t)\cdot \left\|Y^{(1)}\right\|^2_\infty\leq (s-t)\cdot \left\|\frac{\dx}{\dx x^{(1)}}u\right\|^2_\infty \textrm{ a.s.},
  \end{equation*}
  where we used $Z^{(0)}=Y^{(1)}$ a.e.
  
  \textsc{Third summand:} Using \eqref{thisy1} we have
  \begin{align*}
    \bigg|\mathbb{E}\bigg[Y^{(1)}_s-Y^{(1)}_t\bigg]\bigg|
    \leq 2\cdot(t-s)\cdot\left\|\frac{\dx}{\dx x^{(1)}}u\right\|_\infty\cdot\left\|\frac{\dx}{\dx x^{(2)}}u\right\|_\infty\cdot\left\|\frac{\dx}{\dx x^{(1)}}u^{(1)}\right\|_\infty.
  \end{align*}
  
  \textsc{Conclusion:} We have shown
  \begin{equation*}
    \left|\frac{\dx}{\dx x^{(1)}}u(s,x)-\frac{\dx}{\dx x^{(1)}}u(t,x)\right|\leq C |s-t|,
  \end{equation*}
  with some constant $C\in[0,\infty)$, which does not depend on $t,x$ or $s$. In other words $\frac{\dx}{\dx x^{(1)}}u$ is Lipschitz continuous in time. Since it is also Lipschitz continuous in space, it is a Lipschitz continuous function on its whole domain $[0,T]\times\mathbb{R}^2$.

  It remains to show that $\frac{\dx}{\dx x^{(1)}}u$ is $\mathbb{R}\backslash\{0\}$-valued on $[0,1)\times\mathbb{R}^2$:
  Clearly $g'$ is non-negative and does not vanish. Let $t\in [0,1)$, $x\in\mathbb{R}^2$. Consider the associated FBSDE on $[t,1]$. Using \eqref{thisy1} we can write
  \begin{equation*}
    \frac{\dx}{\dx x^{(1)}}u(s,x)=g'(X^{(1)}_T)-\int_t^T Z^{(1)}_r\dx \left(W_r+\int_t^r\left(-2Y^{(1)}_\kappa Y^{(2)}_\kappa\right)\dx\kappa\right). 
  \end{equation*} 
  So there is a probability measure $\mathbb{Q}\sim\mathbb{P}$ such that $\frac{\dx}{\dx x^{(1)}}u(s,x)=\mathbb{E}_{\mathbb{Q}}\big[g'\big(X^{(1)}_T\big)\big]\geq 0$. Now note that $X^{(1)}_T=x^{(1)}+W_T-W_t$ has a non-degenerate normal distribution w.r.t. $\mathbb{P}$. Therefore its distribution is equivalent to the Lebesgue measure. But since $\mathbb{Q}\sim\mathbb{P}$ the distribution of $X^{(1)}_T$ w.r.t. $\mathbb{Q}$ must also be equivalent to the Lebesgue measure. This shows
  \begin{equation*}
    \frac{\dx}{\dx x^{(1)}}u(s,x)=\mathbb{E}_{\mathbb{Q}}\left[g'\left(X^{(1)}_T\right)\right]>0
  \end{equation*}
  since otherwise $g'=0$ a.e. would hold.
\end{proof}

In Theorem \ref{thm:strong} we implicitly solved the Skorokhod embedding problem. To obtain a strong solution is now an immediate consequence.

\begin{corollary}
  Provided $G_0$, $\alpha$ and $\beta$ as in \eqref{eq:gaussian} together with a probability measure $\nu$ such that the corresponding $g$ and $\delta$ fulfil the requirements of Theorem \ref{thm:strong}. Then, there exists a bounded stopping time $\tau$ and a constant $c\in \mathbb{R}$ such that $c+G_{\tau}$ has the law $\nu$.
\end{corollary}

\appendix

\section{BMO-Processes and their properties}\label{sec:appendix}

Let $(\Omega,\mathcal{F}_T,(\mathcal{F}_t)_{t\in[0,T]},\mathbb{P})$ be a complete filtered probability space such that the filtration satisfies the usual hypotheses. Moreover, we assume that there exists a $d$-dimensional Brownian motion $W$ on $[0,T]$ independent of $\mathcal{F}_0$ and that $\mathcal{F}_t=\sigma(\mathcal{F}_0,\mathcal{F}^W_t)$, where $\mathcal{F}^W$ is the natural filtration generated by $W$ and $\mathcal{F}_0$ contains all null sets.

For a probability measure $\mathbb{Q}$ and any $q>0$ and $m\in\mathbb{N}$ we define $\mathcal{H}^q(\mathbb{R}^m,\mathbb{Q})$ as the space of all progressively measurable processes $(Z_t)_{t\in[0,T]}$ with values in $\mathbb{R}^m$ such that $\|Z\|_{\mathcal{H}^q}^q:=\mathbb{E}_\mathbb{Q}\big[\big(\int_0^T|Z_s|^2\dx s\big)^{q/2}\big]<\infty$.

\begin{definition}
  Let $\mathbb{Q}\sim\mathbb{P}$ be an equivalent probability measure and define
  \begin{multline*}
    BMO(\mathbb{Q}):=\big\{Z\colon[0,T]\times\Omega\,\big | \,Z \text{ is progressively measurable and vector-valued } \text{ s.t. } \\
     \exists C\geq 0\,\forall t\in[0,T]: \mathbb{E}_\mathbb{Q}\big[\int_t^T|Z_s|^2\dx s\big|\mathcal{F}_t\big]\leq C \text{ a.s.} \big\}.
  \end{multline*}
  By vector-valued we mean that $Z$ assumes values in some normed vector space. 
\end{definition}
The smallest constant $C$ such that the above bound holds is denoted by $\|Z\|^2_{BMO(\mathbb{Q})}$. For processes $Z\notin BMO(\mathbb{Q})$ we set $\|Z\|_{BMO(\mathbb{Q})}:=\infty$. Furthermore, we call a martingale $M$ a BMO-martingale if $$M_t=M_0+\int_0^t Z_s \dx W_s=:M_0+(Z\bullet W)_t$$ with some $\mathbb{R}^{1\times d}$-valued $Z\in BMO(\mathbb{P})$.

Also, if a progressively measurable process $Z$ is only defined on a subinterval of $[0,T]$, the statement $Z\in BMO(\mathbb{Q})$ means that its natural extension to $[0,T]$, obtained by setting it equal to $0$ everywhere outside its initial domain, is in $BMO(\mathbb{Q})$.

In the following we provide auxiliary results concerning BM0-processes.

\begin{lemma}\label{energ}
  For a probability measure $\mathbb{Q}\sim\mathbb{P}$ let $Z\in BMO(\mathbb{Q})$ be $\mathbb{R}^m$-valued. Then $Z\in\mathcal{H}^{2n}(\mathbb{R}^m,\mathbb{Q})$ for all $n\in\mathbb{N}$ and $\|Z\|_{\mathcal{H}^{2n}(\mathbb{R}^m,\mathbb{Q})}\leq \sqrt[2n]{n!}\,\|Z\|_{BMO(\mathbb{Q})}$.
\end{lemma}

\begin{proof}
  Let us define $A_t:=\int_0^t |Z_s|^2\dx s$, $t\in[0,T]$, which is progressively measurable, non-decreasing, starts at $0$ and satisfies $\mathbb{E}_{\mathbb{Q}}[A_T-A_t|\mathcal{F}_t]\leq \|Z\|^2_{BMO(\mathbb{Q})}$ for all $t\in[0,T]$. Therefore, using energy inequalities we have
  \begin{equation*}
    \mathbb{E}_{\mathbb{Q}}[(A_T)^n]\leq n! \left(\|Z\|^2_{BMO(\mathbb{Q})}\right)^n,
  \end{equation*}
  which implies the assertion.
\end{proof}

\begin{lemma}\label{expzcontrol}
  For all $K>0$ there is a constant $C>0$, which is increasing in $K$, such that
  \begin{equation*}
    \mathbb{E}_{\mathbb{Q}}\bigg[\exp\bigg(\int_t^T |Z_s|\dx s\bigg)\bigg|\mathcal{F}_t\bigg]\leq C\quad\text{a.s., for all }t\in[0,T],
  \end{equation*}
  all probability measures $\mathbb{Q}\sim\mathbb{P}$ and all $Z\in BMO(\mathbb{Q})$ such that $\|Z\|_{BMO(\mathbb{Q})}\leq K$.
\end{lemma}

\begin{proof}
  We apply Lemma \ref{energ} to estimate
  \begin{align*}
    \mathbb{E}_{\mathbb{Q}}\bigg[\exp&\left(\int_t^T |Z_s|\dx s\right)\bigg|\mathcal{F}_t\bigg]
      = \mathbb{E}_{\mathbb{Q}}\left[\sum_{k=0}^\infty\frac{1}{k!}\left(\int_t^T |Z_s|\dx s\right)^k\bigg|\mathcal{F}_t\right]\\
    & \leq \sum_{k=0}^\infty\frac{1}{k!}\mathbb{E}_{\mathbb{Q}}\left[\left(\int_t^T |Z_s|\dx s\right)^k\bigg|\mathcal{F}_t\right]
      \leq \sum_{k=0}^\infty\frac{1} {k!}\mathbb{E}_{\mathbb{Q}}\left[\left((T-t) \int_t^T |Z_s|^2\dx s\right)^\frac{k}{2}\bigg|\mathcal{F}_t\right]\\
    & \leq \sum_{k=0}^\infty\frac{1}{k!}\left(\mathbb{E}_{\mathbb{Q}}\left[\left(T \int_t^T |Z_s|^2\dx s\right)^k\bigg|\mathcal{F}_t\right]\right)^\frac{1}{2}
      \leq \sum_{k=0}^\infty\frac{T^{\frac{k}{2}}}{k!}\left(k!\left(\|Z\|^2_{BMO(\mathbb{Q})}\right)^k\right)^\frac{1}{2}\\
    & \leq \sum_{k=0}^\infty\frac{T^{\frac{k}{2}}}{\sqrt{k!}}K^k=:C<\infty.
  \end{align*}
  We use
  \begin{equation*}
    \left(\frac{T^{\frac{k+1}{2}}}{\sqrt{(k+1)!}}K^{k+1}\right)\cdot\left(\frac{T^{\frac{k}{2}}}{\sqrt{k!}}K^k\right)^{-1}=\frac{T^{\frac{1}{2}}}{\sqrt{k+1}}K\to 0,\quad k\to\infty,
  \end{equation*}
  to see that the series converges absolutely and is monotonically increasing in $K$.
\end{proof}

\begin{lemma}\label{lindimcont}
  For some $N\in\mathbb{N}$ let $Y$ be an $\mathbb{R}^{1\times N}$-valued progressively measurable bounded process on $[0,T]$, the dynamical behavior of which is described by
  \begin{equation}\label{yadb}
    Y_s= Y_T-\int_s^T\dx W^\top_rZ_r-\int_s^T \left(\alpha_r+Y_r\left(\delta_r I_N+\beta_r\right)+\sum_{i=1}^dZ^{i}_r\gamma^{i}_r+\mu_r^\top Z_r\right)\dx r,\quad s\in[0,T],
  \end{equation}
  where
  \begin{itemize}
    \item $Y_T$ is $\mathbb{R}^{1\times N}$-valued, $\mathcal{F}_T$-measurable and bounded,
    \item $Z$ is some $\mathbb{R}^{d\times N}$-valued progressively measurable process s.t. $\int_0^T |Z|^2_r\dx r<\infty$ a.s., which can also be interpreted as a vector $(Z^{i})_{i=1,\ldots, d}$ of $\mathbb{R}^{1\times N}$-valued progressively measurable processes $Z^{i}$, $i=1,\ldots,d$,
    \item $\alpha$ is an $\mathbb{R}^{1\times N}$-valued $BMO(\mathbb{P})$-process,
    \item $\delta$ is some non-negative progressively measurable process with $\int_0^T\delta_s\dx s<\infty$ a.s.,
    \item $I_N\in\mathbb{R}^{N\times N}$ is the identity matrix,
    \item $\beta$ is an $\mathbb{R}^{N\times N}$-valued $BMO(\mathbb{P})$-process,
    \item $\gamma^{i}$, $i=1,\ldots,d$, are progressively measurable and bounded $\mathbb{R}^{N\times N}$-valued processes,
    \item $\mu$ is an $\mathbb{R}^d$-valued $BMO(\mathbb{P})$-process.
  \end{itemize}
  Then $Y$ is bounded by $\|Y\|_{\infty}\leq C_1\cdot\|Y_T\|_\infty+C_2\cdot\|\alpha\|_{BMO(\mathbb{P})}$ with constants $C_1,C_2\in[0,\infty)$ which depend only on $T$, $\|\beta\|_{BMO(\mathbb{P})}$, $\|\mu\|_{BMO(\mathbb{P})}$ and $\|\gamma^{(i)}\|_\infty$, $i=1,\ldots,d,$ and are monotonically increasing in these values.
\end{lemma}

\begin{proof}
  In order to get rid of the term $\mu_r^\top Z_r$ we define a Brownian motion with drift on $[0,T]$ via
  \begin{equation*}
    \tilde{W}_s:=W_s+\int_{0}^s\mu_r\dx r,\quad s\in[0,T]
  \end{equation*}
  Using a standard Girsanov measure change $\tilde{W}$ is a Brownian motion w.r.t. to some equivalent probability measure $\mathbb{Q}$. Furthermore, using \eqref{yadb} the process $Y$ has dynamics
  \begin{equation*}
    Y_s= Y_T-\int_s^T\dx \tilde{W}^\top_rZ_r-\int_s^T \left(\alpha_r+Y_r\left(\delta_r I_N+\beta_r\right)+\sum_{i=1}^dZ^{i}_r\gamma^{i}_r\right)\dx r,\quad s\in[0,T].
  \end{equation*}
  Now, choose a $t\in[0,T]$. We want to control $Y_t$. For that purpose define
  \begin{equation*}
     \Gamma_s:=\exp\left(-\int_{t}^s\left(\delta_rI_N+\beta_r\right)\dx r-\int_{t}^s \sum_{i=1}^d\dx\tilde{W}^i_r\gamma^i_r-\frac{1}{2}\int_{t_1}^s \sum_{i=1}^d \gamma^i_r\gamma^i_r\dx r \right),\quad s\in[t,T].
  \end{equation*}
  According to the It\^o formula $\Gamma$ has dynamics
  \begin{equation*}
    \Gamma_s=\Gamma_T+\int_{s}^T \sum_{i=1}^d\dx\tilde{W}^i_r\gamma^i_r\Gamma_r+\int_s^T\left(\delta_rI_N+\beta_r\right)\Gamma_r\dx r,
  \end{equation*}
  for $s\in[t,T]$. Now, we apply the It\^o formula to $Y_s\Gamma_s$ to obtain
  \begin{align*}
    Y_s\Gamma_s&=Y_T\Gamma_T-\int_s^T\sum_{i=1}^d\dx\tilde{W}^i_r\left(Z^i_r\Gamma_r-Y_r\gamma^i_r\Gamma_r\right) \\
    &-\int_s^T \left\{\left(\alpha_r+ Y_r\left(\delta_rI_N+\beta_r\right)+\sum_{i=1}^d Z^i_r\gamma^i_r\right)\Gamma_r-Y_r\left(\delta_rI_N+\beta_r\right)\Gamma_r-
    \sum_{i=1}^d Z^i_r\gamma^i_r\Gamma_r\right\}\dx r.
  \end{align*}
  A few terms cancel out and we end up with
  \begin{equation}\label{YGdyn}
    Y_s\Gamma_s=Y_T\Gamma_T-\int_s^T\sum_{i=1}^d\dx\tilde{W}^i_r\left(Z^i_r\Gamma_r-Y_r\gamma^i_r\Gamma_r\right)-\int_s^T\alpha_r\Gamma_r\dx r.
  \end{equation}
  We now want to control $\sup_{s\in[t,T]}|\Gamma_s|$: Observe that due to $\delta\geq 0$ we have for all $p\geq 1$
  \begin{align*}
    &\mathbb{E}_{\mathbb{Q}}\left[\sup_{s\in[t,T]}|\Gamma_s|^p\,\bigg|\,\mathcal{F}_t\right]\\
    &=\mathbb{E}_{\mathbb{Q}}\bigg[\sup_{s\in[t,T]}\bigg|\exp\left(-\int_{t}^s\delta_r\dx r -\int_{t}^s\beta_r\dx r-\int_{t}^s \sum_{i=1}^d\dx\tilde{W}^i_r\gamma^i_r-\frac{1}{2}\int_{t}^s \sum_{i=1}^d \gamma^i_r\gamma^i_r\dx r\right)\bigg|^p\bigg|\mathcal{F}_t\Bigg]\\
    &\leq \mathbb{E}_{\mathbb{Q}}\left[\sup_{s\in[t,T]}\left|\exp\left(\left|\int_{t}^s\beta_r\dx r\right|+\left|\int_{t}^s \sum_{i=1}^d\dx\tilde{W}^i_r\gamma^i_r\right|+\frac{1}{2}\left|\int_{t}^s \sum_{i=1}^d \gamma^i_r\gamma^i_r\dx r\right|\right)\right|^p\bigg|\mathcal{F}_t\right] \\
    &\leq\mathbb{E}_{\mathbb{Q}}\Bigg[\sup_{s\in[t,T]}\exp\left(p\int_{t}^s\left|\beta_r\right|\dx r+\frac{p}{2}T\|\gamma\|^2_\infty\right)\cdot
    \sup_{s\in[t,T]}\exp\left(p\left|\int_{t}^s \sum_{i=1}^d\dx\tilde{W}^i_r\gamma^i_r\right|\right)\bigg|\mathcal{F}_t\Bigg],
  \end{align*}
  which using Cauchy-Schwarz inequality can be further controlled by
  \begin{equation*}
    \left(\mathbb{E}_{\mathbb{Q}}\left[\exp\left(\int_{t}^T2p|\beta_r|\dx r+pT\|\gamma\|^2_\infty\right)\bigg|\mathcal{F}_t\right]\mathbb{E}_{\mathbb{Q}}\left[\sup_{s\in[t,T]}\exp\left(2p\left|\int_{t}^s \sum_{i=1}^d\dx\tilde{W}^i_r\gamma^i_r\right|\right)\bigg|\mathcal{F}_t\right]\right)^{\frac{1}{2}}.
  \end{equation*}
  Due to Lemma \ref{expzcontrol} the first of the two factors above can be controlled by a finite constant, which depends only on $p$, $\|\beta\|_{BMO(\mathbb{Q})}$, $\|\gamma\|_\infty$ and $T$ and is monotonically increasing in these values. Notice, that $\|\beta\|_{BMO(\mathbb{Q})}$ can be controlled by $\|\beta\|_{BMO(\mathbb{P})}$ and $\|\mu\|_{BMO(\mathbb{P})}$ by Theorem A.1.6 in \cite{Fromm2015} (or see \cite{Kazamaki1994}, Theorem 2.4. and Theorem 3.6.). 
  
  The second factor can be estimated using Doob's martingale inequality: 
  \begin{align*}
    &\mathbb{E}_{\mathbb{Q}}\left[\exp\left(2p\sup_{s\in[t,T]}\left|\int_{t}^s \sum_{i=1}^d\dx\tilde{W}^i_r\gamma^i_r\right|\right)\bigg|\mathcal{F}_t\right]
    \leq \sum_{k=0}^\infty \frac{1}{k!} \mathbb{E}_{\mathbb{Q}}\left[\left(\sup_{s\in[t,T]}\left|\int_{t}^s \sum_{i=1}^d\dx\tilde{W}^i_r\left(2p\gamma^i_r\right)\right|\right)^k\bigg|\mathcal{F}_t\right] \\
    &\leq  1+\mathbb{E}_{\mathbb{Q}}\bigg[\sup_{s\in[t,T]}\bigg|\int_{t}^s \sum_{i=1}^d2p\dx\tilde{W}^i_r\gamma^i_r\bigg|\bigg|\mathcal{F}_t\bigg]
     +\sum_{k=2}^\infty \frac{1}{k!}\bigg(\frac{k}{k-1}\bigg)^k \mathbb{E}_{\mathbb{Q}}\bigg[\bigg|\int_{t}^T \sum_{i=1}^d\dx\tilde{W}^i_r\big(2p\gamma^i_r\big)\bigg|^k\bigg|\mathcal{F}_t\bigg].
  \end{align*}
  Using Cauchy-Schwarz inequality and Doob's martingale inequality again, the above value can be controlled by
  \begin{align*}
    1+2&\left(\mathbb{E}_{\mathbb{Q}}\left[\left|\int_{t}^T \sum_{i=1}^d\dx\tilde{W}^i_r\left(2p\gamma^i_r\right)\right|^2\bigg|\mathcal{F}_t\right]\right)^{\frac{1}{2}}+\sum_{k=2}^\infty \frac{1}{k!}4 \mathbb{E}_{\mathbb{Q}}\left[\left|\int_{t}^T \sum_{i=1}^d\dx\tilde{W}^i_r\left(2p\gamma^i_r\right)\right|^k\bigg|\mathcal{F}_t\right] \\
    &\leq 10\mathbb{E}_{\mathbb{Q}}\left[\exp\left(2p\left|\int_{t}^T\sum_{i=1}^d\dx\tilde{W}^i_r\gamma^i_r\right|\right)\bigg|\mathcal{F}_t\right].
  \end{align*}
  This value is bounded by a finite constant, which depends only on $p$, $T$ and $\|\gamma\|_\infty$ and is monotonically increasing in these values: For instance use Theorem 2.1 in \cite{Kazamaki1994} by applying it to finitely many sufficiently small subintervals of $[t,T]$ such that $2p\|\gamma\|_\infty$ multiplied by the square root of the size of every subinterval is smaller $1/5$. Also, use the triangle inequality and the tower property after splitting up the stochastic integral.
  One implication of the above control for $\sup_{s\in[t,T]}|\Gamma_s|$ is that the stochastic integral in \eqref{YGdyn} represents a uniformly integrable martingale with respect to $\mathbb{Q}$ since
  \begin{equation*}
    \int_t^s\sum_{i=1}^d\dx\tilde{W}^i_r\left(Z^i_r\Gamma_r-Y_r\gamma^i_r\Gamma_r\right)=Y_s\Gamma_s-Y_t\Gamma_t-\int_t^s\alpha_r\Gamma_r\dx r\quad\textrm{a.s., for all }s\in[t,T],
  \end{equation*}
  and, therefore, using triangle inequality, Cauchy-Schwarz inequality and simple estimates
  \begin{align*}
    \mathbb{E}_{\mathbb{Q}}\bigg[\sup_{s\in[t,T]}\bigg|\int_t^s&\sum_{i=1}^d\dx\tilde{W}^i_r\left(Z^i_r\Gamma_r-Y_r\gamma^i_r|\Gamma_r|\right)\bigg|\bigg]\\
    &\leq 2\|Y\|_\infty\mathbb{E}_{\mathbb{Q}}\left[\sup_{s\in[t,T]}|\Gamma_s|\right]+\mathbb{E}_{\mathbb{Q}}\left[\sup_{s\in[t,T]}\left|\int_t^s\alpha_r\Gamma_r\dx r\right|\right] \\
    &\leq 2\|Y\|_\infty\mathbb{E}_{\mathbb{Q}}\left[\sup_{s\in[t,T]}|\Gamma_s|\right]+\mathbb{E}_{\mathbb{Q}}\left[\sup_{s\in[t,T]}|\Gamma_s|\int_t^T|\alpha_r|\dx r\right] \\
    & \leq 2\|Y\|_\infty\mathbb{E}_{\mathbb{Q}}\left[\sup_{s\in[t,T]}|\Gamma_s|\right]+ \left(\mathbb{E}_{\mathbb{Q}}\left[\sup_{s\in[t,T]}|\Gamma_s|^2\right]\mathbb{E}_{\mathbb{Q}}\left[T\int_t^T|\alpha|^2_r\dx r\right]\right)^{\frac{1}{2}},
  \end{align*}
  which is finite due to $\alpha\in BMO(\mathbb{P})$ and Theorem A.1.6 in \cite{Fromm2015}. We can finally estimate using \eqref{YGdyn} and Cauchy-Schwarz inequality:
  \begin{align*}
    \left|Y_{t}\right|&=\left|\mathbb{E}_{\mathbb{Q}}\left[Y_{t}\Gamma_{t}|\mathcal{F}_t\right]\right|=\left|\mathbb{E}_{\mathbb{Q}}\left[Y_{T}\Gamma_{T}|\mathcal{F}_t\right]-\mathbb{E}_{\mathbb{Q}}\left[\int_{t}^T\alpha_r\Gamma_r\dx r\bigg|\mathcal{F}_t\right]\right|\\
    &\leq \|Y_T\|_\infty\mathbb{E}_{\mathbb{Q}}\left[|\Gamma_{T}||\mathcal{F}_t\right]+\left(\mathbb{E}_{\mathbb{Q}}\left[\sup_{s\in[t,T]}|\Gamma_s|^2\bigg|\mathcal{F}_t\right]\right)^{\frac{1}{2}}\left(\mathbb{E}_{\mathbb{Q}}\left[T\int_{t}^T|\alpha_r|^2\dx r\bigg|\mathcal{F}_t\right]\right)^{\frac{1}{2}}\\
    &\leq\|Y_T\|_\infty\sqrt{\mathbb{E}_{\mathbb{Q}}\left[|\Gamma_{T}|^2|\mathcal{F}_t\right]}+\sqrt{T}\|\alpha\|_{BMO(\mathbb{Q})}
    \left(\mathbb{E}_{\mathbb{Q}}\left[\sup_{s\in[t,T]}|\Gamma_s|^2\bigg|\mathcal{F}_t\right]\right)^{\frac{1}{2}}\\
    &\leq\|Y_T\|_\infty\sqrt{\mathbb{E}_{\mathbb{Q}}\left[|\Gamma_{T}|^2|\mathcal{F}_t\right]}+K_1\|\alpha\|_{BMO(\mathbb{P})}
    \left(\mathbb{E}_{\mathbb{Q}}\left[\sup_{s\in[t,T]}|\Gamma_s|^2\bigg|\mathcal{F}_t\right]\right)^{\frac{1}{2}},
  \end{align*}
  where we again used Theorem A.1.6 in \cite{Fromm2015}. $K_1$ depends only on $\|\mu\|_{BMO(\mathbb{P})}$ and $T$.
\end{proof}

The following theorem is an extension of a result from \cite{prhedg}.

\begin{thm}\label{BSDEBMO}
  Let $Y$, $Z$, $X$, $\psi$, $\varphi$ be some progressively measurable processes on $[0,T]$ such that $Y$ is real-valued and bounded, $Z$ is $\mathbb{R}^{1\times d}$-valued with $\int_0^T|Z_s|^2<\infty$ a.s., $\psi$ and $\varphi$ are real-valued and in $BMO(\mathbb{P})$, and $X$ is real-valued and satisfies $X\leq \psi^2+|Z|\varphi+C|Z|^2$ with some constant $C>0$.
  Furthermore,  suppose that
  \begin{equation*}
    Y_t=Y_T+\int_t^T X_s\dx s-\int_t^T Z_s\dx W_s \quad \textrm{a.s., }\quad  t\in[0,T].
  \end{equation*}
  Then we have $\|Z\|_{BMO(\mathbb{P})}\leq K<\infty$ for some constant $K$, which only depends on $\|Y\|_\infty$, $C$, $\|\varphi\|_{BMO(\mathbb{P})}$, $\|\psi\|_{BMO(\mathbb{P})}$ and is monotonically increasing in theses values.
\end{thm}

\begin{proof}
  Clearly, we see $X\leq\psi^2+|Z|\varphi+C|Z|^2\leq (\psi^2+\frac{1}{2}\varphi^2)+(C+\frac{1}{2})|Z|^2$. Let us set $\tilde{\psi}:=(\psi^2+\frac{1}{2}\varphi^2)^{1/2}\in BMO(\mathbb{P})$, $\tilde{C}:=C+\frac{1}{2}$, and write
  \begin{equation*}
    Y_t=Y_0-\int_0^t X_s\dx s+\int_0^t Z_s\dx W_s.
  \end{equation*}
  Let $\beta\in\mathbb{R}$ be some constant specified later. Using It\^o's formula we get
  $$ \exp(\beta Y_t)=\exp(\beta Y_0)-\int_0^t \beta\exp(\beta Y_s)X_s\dx s+\int_0^t \beta\exp(\beta Y_s)Z_s\dx W_s+
  \frac{\beta^2}{2}\int_0^t \exp(\beta Y_s)|Z_s|^2\dx s. $$
  So for every stopping time $\tau\in [t,T]$ we can write
  \begin{equation*}
    \exp(\beta Y_t)=\exp(\beta Y_\tau)+\int_t^\tau \beta\exp(\beta Y_s)X_s\dx s-\int_t^\tau \beta\exp(\beta Y_s)Z_s\dx W_s- \frac{\beta^2}{2}\int_t^\tau \exp(\beta Y_s)|Z_s|^2\dx s,
  \end{equation*}
  which can be rearranged to
  \begin{align*}
    \beta\int_t^\tau \exp(\beta Y_s)\left(\frac{\beta}{2}|Z_s|^2-X_s\right)\dx s
    =\exp(\beta Y_\tau)-\exp(\beta Y_t)-\int_t\tau \beta\exp(\beta Y_s)Z_s\dx W_s,
  \end{align*}
  or again to
  \begin{align*}
    \beta\int_t^\tau \exp(\beta Y_s)&\left(\frac{\beta}{2}|Z_s|^2+\tilde{\psi}^2_s-X_s\right)\dx s\\
    &=\exp(\beta Y_\tau)-\exp(\beta Y_t)+\beta\int_t^\tau \exp(\beta Y_s)\tilde{\psi}^2_s\dx s-\int_t^\tau \beta\exp(\beta Y_s)Z_s\dx W_s.
  \end{align*}
  Setting $\beta:=2\tilde{C}+2=2C+3$, we have $|Z_s|^2\leq\frac{\beta}{2}|Z_s|^2+\tilde{\psi}^2_s-X_s$. Now choose a localizing sequence of stopping times $\tau_n\in[t,T]$, $n\in\mathbb{N}$, such that $\mathbb{E}\left[\int_t^{\tau_n}|Z_s|^2\dx s\right]<\infty$ for every $n\in\mathbb{N}$ while $\tau_n\uparrow T$ for $n\to\infty$. Applying conditional expectations we have
  \begin{align*}
    \mathbb{E}\left[\beta\int_t^{\tau_n} \exp(\beta Y_s)|Z_s|^2\dx s\bigg|\mathcal{F}_t\right]
    &\leq \mathbb{E}\left[\beta\int_t^{\tau_n} \exp(\beta Y_s)\left(\frac{\beta}{2}|Z_s|^2+\tilde{\psi}^2_s-X_s\right)\dx s\right]\\
    \leq \mathbb{E}&\left[\exp(\beta Y_{\tau_n})-\exp(\beta Y_t)+\beta\int_t^{\tau_n} \exp(\beta Y_s)(\psi^2+\frac{1}{2}\varphi^2)\dx s\bigg|\mathcal{F}_t\right],
  \end{align*}
  which we can rewrite as
  \begin{align*}
    \mathbb{E}&\bigg[\int_t^{\tau_n} \exp(\beta Y_s)|Z_s|^2\dx s\bigg|\mathcal{F}_t\bigg]\\
    &\leq\mathbb{E}\left[\frac{\exp(\beta Y_{\tau_n})-\exp(\beta Y_t)}{\beta Y_T-\beta Y_t}\left(Y_{\tau_n}-Y_t\right)+\int_t^{\tau_n} \exp(\beta Y_s)(\psi^2+\frac{1}{2}\varphi^2)\dx s\bigg|\mathcal{F}_t\right]\\
    &\leq \left\|\frac{\exp(\beta Y_{\tau_n})-\exp(\beta Y_t)}{\beta Y_{\tau_n}-\beta Y_t}\right\|_\infty\cdot\|Y_{\tau_n}-Y_t\|_\infty+\exp\left(\beta\|Y\|_\infty\right)\left(\|\psi\|^2_{BMO(\mathbb{P})}+\frac{1}{2}\|\varphi\|^2_{BMO(\mathbb{P})}\right).
  \end{align*}
  Finally, note that the exponential function is Lipschitz continuous on any interval $[a,b]$ with $\exp(a\vee b)$ as Lipschitz constant, so
  \begin{equation*} \left\|\frac{\exp(\beta Y_{\tau_n})-\exp(\beta Y_t)}{\beta Y_{\tau_n}-\beta Y_t}\right\|_\infty\cdot\|Y_{\tau_n}-Y_t\|_\infty\leq \exp(\beta \|Y\|_\infty)\cdot 2\cdot\|Y\|_\infty.
  \end{equation*}
  We obtain by monotone convergence
  \begin{align*}
    \mathbb{E}&\left[\int_t^T|Z_s|^2 \dx s\bigg|\mathcal{F}_t\right]=\lim_{n\to\infty}\mathbb{E}\left[\int_t^{\tau_n}|Z_s|^2\dx s\bigg|\mathcal{F}_t\right]\\
    &\leq \lim_{n\to\infty}\exp(\beta \|Y\|_\infty)\mathbb{E}\left[\int_t^{\tau_n} \exp(\beta Y_s)|Z_s|^2\dx s\bigg|\mathcal{F}_t\right]\\
    &\leq 2\exp(2\beta \|Y\|_\infty)\|Y\|_\infty+\exp\left(2\beta\|Y\|_\infty\right)\left(\|\psi\|^2_{BMO(\mathbb{P})}+\frac{1}{2}\|\varphi\|^2_{BMO(\mathbb{P})}\right)\\
    &= 2\exp(2(2C+3) \|Y\|_\infty)\|Y\|_\infty+\exp\left(2(2C+3)\|Y\|_\infty\right)\left(\|\psi\|^2_{BMO(\mathbb{P})}+\frac{1}{2}\|\varphi\|^2_{BMO(\mathbb{P})}\right),
  \end{align*}
  which is finite and increasing in $\|Y\|_\infty$, $C$, $\|\varphi\|_{BMO(\mathbb{P})}$ and $\|\psi\|_{BMO(\mathbb{P})}$.
\end{proof}

\subsection*{Miscellaneous}

Finally, we collect here elementary properties of weakly differentiable processes.

\begin{lemma}(\cite{Fromm2015}, Lemma 4.3.11)\label{secdivcon}
  Let $\varphi\colon\mathbb{R}\to\mathbb{R}$ be twice weakly differentiable such that $\varphi(0)=0$ and $\|\varphi''\|_\infty<\infty$. Then one has
  \begin{equation*}
    \left|\int_\mathbb{R}\varphi(\sigma\cdot z)\frac{1}{\sqrt{2\pi}}e^{-\frac{1}{2}z^2}\dx z\right|\leq \frac{1}{2}\sigma^2\|\varphi''\|_\infty \quad \text{for all } \sigma\in [0,\infty).
  \end{equation*}
\end{lemma}

\begin{lemma}(\cite{Ambrosio1990}, Corollary 3.2, and \cite{Fromm2015}, Lemma A.3.1)\label{chainruleambr}
  Let $N,m\in\mathbb{N}$ and let $g\colon\mathbb{R}^N\to \mathbb{R}^m$ be Lipschitz continuous. Moreover, let $X\colon\mathbb{R}^n\to \mathbb{R}^{N}$, $n\in\mathbb{N}$ be weakly differentiable. Then one has that $g(X)$ is also weakly differentiable, the restriction $g|_{T^{X}_\lambda}$ of $g$ to the affine space
  \begin{equation*}
    T^{X}_\lambda:=\left\{x\in\mathbb{R}^N\,\bigg|\,x=X(\lambda)+\frac{\dx}{\dx\lambda}X(\lambda)v,\text{ for some }v\in\mathbb{R}^n\right\}
  \end{equation*}
  is differentiable at $X(\lambda)$ for almost every $\lambda\in\mathbb{R}^n$, and
  \begin{equation*}
    \frac{\dx}{\dx\lambda}g(X)(\lambda)=\frac{\dx}{\dx x}g|_{T^{X}_\lambda}(X(\lambda))\frac{\dx}{\dx\lambda}X(\lambda)\quad \text{for almost all } \lambda\in\mathbb{R}^n.
  \end{equation*} 
  
  In particular, this implies:
  \begin{itemize}
   \item If $n=N$ and the matrix $\frac{\dx}{\dx\lambda}X(\lambda)$ is invertible for a.a. $\lambda$, then $T^{X}_\lambda=\mathbb{R}^N$ for a.a. $\lambda$ and
          $\frac{\dx}{\dx\lambda}g(X)=\left(\frac{\dx}{\dx x}g\right)(X)\frac{\dx}{\dx\lambda}X$
          a.e., where $\frac{\dx}{\dx x}g$ is a weak derivative of $g$.
    \item If $g$ is differentiable everywhere then $\frac{\dx}{\dx\lambda}g(X)=\left(\frac{\dx}{\dx x}g\right)(X)\frac{\dx}{\dx\lambda}X$ a.e.
    \item If $g$ is only locally Lipschitz continuous rather than Lipschitz continuous, but differentiable everywhere, while $X$ is bounded, then still $\frac{\dx}{\dx\lambda}g(X)=\left(\frac{\dx}{\dx x}g\right)(X)\frac{\dx}{\dx\lambda}X$ a.e.
  \end{itemize}
\end{lemma}

{\small
\bibliography{quellenSkoro}{}
\bibliographystyle{amsalpha}
}

\end{document}